\documentclass{article}

\usepackage[T1]{fontenc}
\usepackage[utf8]{inputenc}
\usepackage{amsmath, amssymb, amsthm, latexsym, amsfonts, ifsym, bbm, enumerate, url, mathrsfs}

\newtheorem{theorem}{Theorem}
\newtheorem{lemma}[theorem]{Lemma}
\newtheorem{defn}[theorem]{Definition}

\newtheorem{prop}[theorem]{Proposition} 
\newtheorem*{prop1}{Proposition \ref{rootnlemma} (Restatement)}
\newtheorem*{prop2}{Proposition \ref{sbiglemma} (Restatement)}
\newtheorem*{lem14}{Lemma \ref{schrammshrinklemma} (Restatement)}
\newtheorem*{lem15}{Lemma \ref{mshrinklemma} (Restatement)}
\newtheorem*{lem16}{Lemma \ref{mgrowthlemma} (Restatement)}
\newtheorem*{lem17}{Lemma \ref{sshrinklemma} (Restatement)}
\newtheorem*{lem18}{Lemma \ref{sgrowthlemma} (Restatement)}
\theoremstyle{definition} 

\newtheorem{rmk}[theorem]{Remark}
\newtheorem{example}[theorem]{Example}

\newcommand{\prob}{\mathbb{P}}
\newcommand{\expect}{\mathbb{E}}

\numberwithin{equation}{section}

\title{A Coupling Argument for the Random Transposition Walk}
\author{Olena Blumberg}
\date{}
\begin{document}
\maketitle

\begin{abstract}
This paper explores the mixing time of the random transposition walk on the symmetric group $S_n$. While it has long been known that this walk mixes in $O(n \log n)$ time, this result has not previously been attained using coupling. A coupling argument showing the correct order mixing time is presented. This is accomplished by first projecting to conjugacy classes, and then using the Bubley-Dyer path coupling construction. In order to obtain appropriate bounds on the time it takes the path coupling to meet, ideas from Schramm's paper ``Compositions of Random Transpositions'' are used. 
\end{abstract}

\section{Introduction}

This paper studies the random transposition walk on the symmetric group $S_n$ -- in card shuffling terms, the possible permutations of a deck of $n$ cards. Here's a description of the random walk: lay $n$ cards out in a row, and pick one card uniformly with your right hand, and another card independently uniformly with your left hand (note that you may have picked the same card.) Then, swap the cards -- this is an extremely simple shuffling scheme for $n$ cards.

Below, we study the {\it mixing time} of the above random walk: that is, the number of shuffles that it takes to thoroughly mix up the deck (see Section \ref{backgroundanddefinitions} for a precise definition.) To be more specific, a coupling argument demonstrating that the mixing time of the random transposition walk is $O(n \log n)$ is presented. Coupling is an intuitive probabilistic technique that bounds mixing time in the following way: define a process $(X_t, Y_t)_{t\geq 0}$ such that both $(X_t)_{t\geq 0}$ and $(Y_t)_{t\geq 0}$ are Markov chains with the same transition matrix, but with $X_t$ starting at $x$ and $Y_t$ starting at $y$. As will be described more precisely in Section \ref{backgroundanddefinitions} below, the goal is to have the two chains meet: by the time that this has happened with high probability for every choice of $x$ and $y$, it can be shown that the Markov chain has mixed. This technique is usually traced back to Doeblin \cite{DoeblinCoupling}; two good reference books which illustrate its many uses are Lindvall's ``Lectures on the coupling method'' \cite{Lindvall} and Thorisson's ``Coupling, stationarity, and regeneration'' \cite{Thorisson}.

The existence of a coupling argument showing an $O(n \log n)$ mixing time is a long-standing open problem. Due to its simplicity and symmetry, the random transposition walk was one of the first ones considered in burgeoning field of Markov chains mixing times. As noted in \cite{DiaconisandShah}, the mixing time of this walk was first bounded by Aldous in 1980, who showed that it must be between order $n$ and $n^2$ and conjectured that it must be of order $n \log n$. This was proved in 1981 in ``Generating a random permutation with random transpositions'' by Diaconis and Shahshahani \cite{DiaconisandShah}. This paper uses Fourier analysis to show that the walk experiences a cut-off, mixing in a window of order $n$ around time $\frac{1}{2} n \log n$. 

The beautiful proof in \cite{DiaconisandShah} uses the tools of representation theory and Fourier analysis, and hence is non-probabilistic. While a purely probabilistic strong stationary time proof for an $O (n \log n)$ mixing time was discovered by Broder in 1985 \cite{BroderStrongStationary}, a coupling argument proved to be more elusive. The main difficulty is due to the fact that a Markovian coupling cannot succeed; indeed, Lemma \ref{Markoviancouplingfail} below shows that such an approach can never prove a bound of order better than $n^2$. It has been shown by Griffeath \cite{GriffeathMaximalCoupling} and then Pitman \cite{PitmanMaximalCoupling} that a maximal coupling must exist, but it evidently has to be non-Markovian. There has been continued interest in finding such a coupling -- for example, Peres named it as an interesting open problem in \cite{YuvalBCNotes}. This paper resolves this problem. (Another approach for finding such a non-Markovian coupling can be seen in the preprint ``Mixing times via super-fast coupling'' \cite{kovchegovburton}.)

This question is approached here by first projecting the random transposition walk to conjugacy classes. T
his projection is also a Markov chain, called a split-merge random walk \cite{SchrammLargeCycles}. Using the fact that the random transposition walk started from the identity is constant on conjugacy classes, it suffices to find the mixing time of the split-merge random walk. The path coupling technique of Bubley and Dyer \cite{BubleyDyer} is used to examine the split-merge random walk. However, this is not straightforward, since in the worst case scenario, the split-merge random walk contracts by only $1 - \frac{1}{n^2}$.  

It is shown here that on average, the split-merge random walk does indeed contract by $1 - \frac{1}{n}$, enabling the use of path coupling to conclude that the walk mixes in $O(n \log n)$ time. This argument does not, however, show cut-off: indeed, as noted in Remark \ref{constantlarge} below, the constant in front of the $n \log n$ is very large. To show that the contraction coefficient is of the right order, the techniques of Schramm from his paper ``Compositions of random transpositions'' \cite{SchrammLargeCycles} are used. He shows that large cycles in the random transposition walk emerge after time $\frac{n}{2}$, and then proves the law for the scaled cycles. Methods from ``Compositions of random transpositions'' have given rise to the wonderful paper ``Mixing times for random k-cycles and coalescence-fragmentation chains'' by Berestycki, Schramm, and Zeitouni \cite{BerestyckiZeitouniSchramm}, which uses probabilistic techniques to get the right answer for a generalization of the random transposition walk. 

\section{Background and Definitions}\label{backgroundanddefinitions}

Before stating the main result of this paper, a number of definitions are necessary. If $\mu$ and $\nu$ are two probability distributions on a finite state space $\Omega$, then the total variation distance between $\mu$ and $\nu$ is defined to be $\left\| \mu - \nu \right\|_{TV} = \frac{1}{2}\sum_{x\in \Omega} | \mu(x) - \nu(x)|$. For a Markov chain with transition probabilities $Q(x, y)$ and stationary distribution $\pi$, the total variation distance at time $t$ is defined to be $d(t) =  \left\| Q^t(x, \cdot) - \pi\right\|_{TV}$ and the mixing time is
\begin{equation*}
\tau_{\mathrm{mix}}(\epsilon) = \min\left\{t \left| \right. d(t) \leq \epsilon\right\}
\end{equation*}
Conventionally, $\tau_{\mathrm{mix}}$ is defined to be $\tau_{\mathrm{mix}}(1/4)$. 

A coupling of a pair of Markov chains both with transition matrix $Q$ is a process $(X_t, Y_t)_{t\geq 0}$ such that both $(X_t)_{t\geq 0}$ and $(Y_t)_{t\geq 0}$ are Markov chains with transition matrix $Q$, but which might have different starting distributions. The coupling inequality (Corollary 5.3 in \cite{YuvalBook}) states that if $(X_t, Y_t)$ is a coupling of a pair of Markov chains such that $X_0 = x$ and $Y_0 = y$, and $T_{x, y}$ is a random time at which the chains have met, then 
\begin{equation*}
d(t) \leq \max_{x, y} \mathbb{P} \left\{ T_{x,y} > t \right\}
\end{equation*}
The above inequality allows coupling to be used to bound mixing times. It is now possible to state the main result of this paper: 

\begin{theorem}\label{schrammcouplingtheorem} 
There exists a coupling argument that shows that the random transposition walk on $S_n$ mixes in time of order $n \log n$: that is, it demonstrates that there exists a constant $C$ such that 
\begin{equation*}
\tau_{\mathrm{mix}} \leq C n\log n 
\end{equation*}   
\end{theorem} 

Before launching into the proof, it is instructive to consider the many ways an $O(n \log n)$ mixing time has been obtained for this walk, as well as the uses of the result. This bound was first obtained by Diaconis and Shahshahani in \cite{DiaconisandShah}. This result is beautiful and extremely precise; however, the scope of the technique is limited as it requires fully diagonalizing the random walk. While this is possible for a number of walks, including walks that are not random walks on groups, this is a drawback to the method. This result is also extremely useful for comparison theory. As shown by Diaconis and Saloff-Coste in \cite{DiaconisSaloffComparisonGroups}, the Dirichlet form can be used to compare all the eigenvalues of the chain, resulting in good bounds for a variety of walks. For example, Jonasson uses this result in \cite{JonassonOverlapping} to show that the overlapping cycle shuffle mixes in $O(n^3 \log n)$ time.

As noted above, the first probabilistic proof of the result was by Broder \cite{BroderStrongStationary} and used {\it strong stationary times}: stopping times $T$ such that the conditional distribution of $X_T$ given $T$ is stationary. Since the stationary distribution for the random transposition walk is uniform, this is equivalent to stating that for all $\sigma \in S_n$ and all positive integers $k$, 
\begin{equation*}
\prob(X_T = \sigma \left| \right. T = k ) = \frac{1}{n!} 
\end{equation*}  
The following is Broder's strong stationary time argument, as summarized in Chapter 9 of \cite{YuvalBook}. Let $R_t$ and $L_t$ be the cards chosen by the right and left hand, respectively. Start the process with no marked cards, and use the following marking scheme: at each step, mark a card $R_t$ if $R_t$ is unmarked, and either (a) $L_t$ is marked or (b) $R_t = L_t$. Define the stopping time $T$ to be the first time all $n$ cards are marked. It is easy to show that this is indeed a strong stationary time, and that $T$ is around $2n \log n$. This argument provides an $O(n \log n)$ mixing time, but not the correct constant. It was improved by Matthews \cite{MatthewsStationaryTranspositions} in 1988 by creating a more complicated rule for marking the cards. This argument showed a cut-off for the walk at time $\frac{1}{2} n \log n$. These arguments are probabilistic and intuitive, and elucidate the reasons for the mixing time in a way that Fourier analysis does not. However, they are heavily reliant on the symmetry of the random transposition walk and as such are difficult to generalize. 

The recent paper by Berestycki, Schramm and Zeitouni \cite{BerestyckiZeitouniSchramm} uses a different approach. Their technique provides the correct answer for the following generalization of the Markov chain: instead of using a uniformly chosen random transposition at each step, a random $k$-cycle is used. This paper obtains the correct $\frac{1}{k} n \log n$ answer for any fixed $k$. Like this paper, they begin by projecting the walk to conjugacy classes and then make use of the results of Schramm in \cite{SchrammLargeCycles}. The tools of both this result and Schramm's original paper are graph theoretic: for example, a transposition is considered to be an edge in a random graph process on $n$ vertices. Unfortunately, this exciting method again requires considerable symmetry, since the projection to conjugacy classes has to be a Markov chain. This is also a drawback of the coupling approach which is presented here. 

Another intriguing technique explored by Burton and Kovchegov \cite{kovchegovburton} uses non-Markovian coupling. While I have found the ideas in this paper difficult, the approximate approach is that the standard coupling argument by Aldous which results in $O(n^2)$  bound can be improved by `looking into the future.' A non-Markovian argument with a somewhat similar flavor has previously been implemented for the coloring chain by Hayes and Vigoda \cite{VigodaNonMarkovian}. Here's a very approximate sketch of the idea for random transpositions: say that a pair $(\sigma, \tau)$ in $S_n$ currently differs in the cards labeled $i$ and $j$. The standard coupling for this pair transposes the cards with the same labels in both $\sigma$ and $\tau$, unless the next transposition is $(i, j)$. However, it is possible to do something different: if the next step transposes cards labeled $i$ and $k$ in $\sigma$, the next step in $\tau$ can transpose either cards labeled $i$ and $k$ or cards labeled $j$ and $k$. If the coupling is Markovian, then the choice makes no difference; however, `looking into the future' can substantially improve the bounds. In work stemming from an unrelated project, I hope to show this for a number of different walks in an upcoming paper. 

The argument in this paper proceeds by projecting the walk to conjugacy classes.  It is a well-known result that the conjugacy classes of $S_n$ are indexed by partitions of $n$. Recall that a partition of $n$ is an $m$-tuple $(a_1, a_2, \dots, a_m)$ of positive integers that sum to $n$, where $m$ can be any integer, and $a_1 \geq a_2 \geq \dots \geq a_m$. Let $\mathcal{P}_n$ be the set of partitions of $n$. The projection of the random transposition walk on $S_n$ to conjugacy classes is also a Markov chain, called a {\it split-merge} random walk. It is often referred to as a coagulation-fragmentation chain, and it has been extensively studied -- see \cite{DiaconisZeitouniPD1} for some references. 

\begin{defn}\label{splitmergedesc}
Assume the random walk is currently at partition $(a_1,\dots, a_m)$. Then, there are three possibilities for the next move: either merge a pair of parts, split a part into two pieces, or stay in place. (All of these moves are followed by rearranging the new parts to be in non-decreasing order.) 
\begin{itemize}
\item {\bf Split:} A pair $a_i$ can be replaced by the pair $(r, a_i-r)$. For each $r$ between $1$ and $a_i-1$, the probability of this particular split is $\frac{a_i}{n^2}$. 

Note that this phrasing takes the order into account: here, a more convenient phrasing is the following: for each $r < \frac{a_i}{2}$, split $a_i$ into $\{r, a_i - r\}$ with probability $\frac{2a_i}{n^2}$. If $a_i$ is even and $r = \frac{a_i}{2}$, split $a_i$ into $\{r, a_i-r\}$ with probability $\frac{a_i}{n^2}$. 

\item {\bf Merge:} Replace the parts $a_i$ and $a_j$ by $a_i + a_j$. This is done with probability $\frac{2a_i a_j}{n^2}$. 

\item {\bf Stay in Place:} Stay at the partition $(a_1, a_2, \dots, a_m)$ with probability $\frac{1}{n}$. 
\end{itemize}
\end{defn}

\begin{example}
Here is an example of the split-merge random walk. Let $n = 5$, and assume the walk is currently at $(4,1)$. Then, the next step $X_1$ is distributed as follows: 
\begin{equation*}
X_1 = 
\begin{cases}
(5) & \text{with probability } \frac{8}{25}\\
(4,1) & \text{with probability } \frac{1}{5}\\
(3,1,1) & \text{with probability } \frac{8}{25}\\
(2,2,1) & \text{with probability } \frac{4}{25}
\end{cases} 
\end{equation*}
\end{example} 

The primary walk under consideration is the split-merge random walk, but for some of the proofs, the original transposition walk is needed. With that in mind, make the following two definitions: 

\begin{defn}\label{Cycdefn}
For $\alpha \in S_n$, define $\mathrm{Cyc}(\alpha)$ to be the partition corresponding to the cycle type of $\alpha$. For $\sigma \in \mathcal{P}_n$, let
\begin{equation*}
\mathrm{Perm}(\sigma) = \{ \alpha \in S_n \left| \right. \mathrm{Cyc}(\alpha) = \sigma\}
\end{equation*}
be the set of all permutations with cycle type $\sigma$. 
\end{defn}

\begin{defn}\label{transpositionwalknotation}
Let $(X_t)_{t\geq 0}$ denote the split-merge random walk, and let $(\bar{X}_t)_{t\geq 0}$ denote the random transposition walk, so that for all $t$, 
\begin{equation*}
X_t = \mathrm{Cyc}(\bar{X}_t)
\end{equation*}
Furthermore, let $P$ and $\pi$ be the transition matrix and stationary distribution for $(X_t)_{t\geq 0}$, respectively, and define $\bar{P}$ and $\bar{\pi}$ analogously for $(\bar{X}_t)_{t\geq 0}$.
\end{defn}

The next argument shows it  suffices to consider the split-merge random walk. The following proof take a little bit of space to write down, but is actually very simple -- the key idea is that the random transposition walk started at the identity is always uniformly distributed over each conjugacy class. (This also follows from a more general result -- see Chapter 3F of \cite{DiaconisBook}.)

\begin{lemma}\label{equivalencesplittrans}
Let $P, \bar{P}, \pi$ and $\bar{\pi}$ be defined as in Definition \ref{transpositionwalknotation} above. Then,  
\begin{equation*}
\max_{\alpha \in S_n} \left\| \bar{P}^t(\alpha, \cdot) - \bar{\pi} \right\|_{TV} \leq  \max_{\sigma \in \mathcal{P}_n} \left\| P^t(\sigma, \cdot) - \pi \right\|_{TV} 
\end{equation*} 
\end{lemma}
\begin{proof}[\bf Proof:]
Since the random transposition walk is a random walk on a group, it's vertex transitive. Therefore, for all $\alpha \in S_n$, 
\begin{equation*}
\left\|\bar{P}^t(\alpha, \cdot) - \bar{\pi} \right\|_{TV}= \left\|\bar{P}^t(id, \cdot) - \bar{\pi} \right\|_{TV}
\end{equation*}
where $id$ is the identity permutation. Thus, it suffices to show that 
\begin{equation*}
\left\|\bar{P}^t(id, \cdot) - \bar{\pi} \right\|_{TV} \leq  \max_{\sigma \in \mathcal{P}_n} \left\| P^t(\sigma, \cdot) - \pi \right\|_{TV} 
\end{equation*} 
Now, let $\sigma_0 = \text{Cyc}(id) = (1,1,\dots, 1)$. It suffices to show that
\begin{equation}\label{equaldistancetostationarity} 
\left\|\bar{P}^t(id, \cdot) - \bar{\pi} \right\|_{TV} = \left\| P^t(\sigma_0, \cdot) - \pi \right\|_{TV} 
\end{equation}

Since the split-merge random walk is a projection of the random transposition walk, for $\sigma \in \mathcal{P}_n$, 
\begin{equation} \label{piissymmetric}
\pi(\sigma) = \sum_{\alpha \in \mathrm{Perm}(\sigma)} \bar{\pi}(\alpha) =\frac{\left| \mathrm{Perm}(\sigma)\right|}{n!}
\end{equation}
since $\bar{\pi}$ is the uniform distribution on $S_n$. Similarly, 
\begin{equation*}
P^t(\sigma_0, \sigma) = \sum_{\alpha \in \mathrm{Perm}(\sigma)} \bar{P}^t(id, \alpha)
\end{equation*} 
Furthermore, note that both the identity permutation and the the random transposition walk are symmetric with respect to $\{1, 2, \dots, n\}$. Hence for any $\alpha_1, \alpha_2$ with the same cycle structure, $\bar{P}^t(id, \alpha_1) = \bar{P}^t(id, \alpha_2)$ for all $t$. Combining this with the equation above shows that for $\alpha \in \mathrm{Perm}(\sigma)$, 
\begin{equation}\label{Pissymmetric}
P^t(\sigma_0, \sigma) = \left|\mathrm{Perm}(\sigma)\right| \bar{P}^t(id, \alpha) 
\end{equation}
Using Equations \eqref{piissymmetric} and \eqref{Pissymmetric}, 
\begin{equation*}
\sum_{\alpha \in \mathrm{Perm}(\sigma)}  \left| \bar{P}^t(id, \alpha) - \frac{1}{n!} \right| = \left| P^t(\sigma_0, \sigma) - \pi(\sigma) \right| 
\end{equation*} 
Finally, putting all this together, 
\begin{align*}
2\left\|\bar{P}^t(id, \cdot) - \bar{\pi} \right\|_{TV} &= \sum_{\alpha \in S_n} \left| \bar{P}^t(id, \alpha) - \frac{1}{n!}\right| = \sum_{\sigma \in \mathcal{P}_n} \sum_{\alpha \in \mathrm{Perm}(\sigma)}  \left| \bar{P}^t(id, \alpha) - \frac{1}{n!} \right|\\
&= \sum_{\sigma \in \mathcal{P}_n} \left| P^t(\sigma_0, \sigma) - \pi(\sigma) \right| = 2\left\| P^t(\sigma_0, \cdot) - \pi \right\|_{TV}
\end{align*}
which proves Equation \eqref{equaldistancetostationarity}, as desired. \end{proof}
\begin{rmk}
Although it is not needed, it is very easy to use the triangle inequality to prove the opposite inequality to the one in Lemma \ref{equaldistancetostationarity}. Hence, the two quantities are actually equal. 
\end{rmk}

Before proceeding to sketch the upcoming proof, it is shown that a Markovian coupling for the random transposition walk cannot hope to give an $O(n \log n)$ mixing time. 

\begin{lemma}\label{Markoviancouplingfail}
A Markovian coupling $(\bar{X}_t, \bar{Y}_t)$ of the random transposition walk takes at least $\Omega(n^2)$ time to meet. 
\end{lemma}

\begin{proof}[\bf Proof:]
It easy to check that wherever the two random transposition walks currently are, if $\bar{X}_t \neq \bar{Y}_t$, then 
\begin{equation*}
\prob\left(\bar{X}_{t+1} = \bar{Y}_{t+1} \right) \leq \frac{6}{n^2} 
\end{equation*} 
To verify this, note that if $\bar{X}_t$ and $\bar{Y}_t$ differ only in the transposition $(i, j)$, then the only way to meet is to transpose $i$ and $j$ in one of them, and to stay in place in the other one; similar arguments hold if $\bar{X}_t$ and $\bar{Y}_t$ are two transpositions apart, and in all other cases, the probability of meeting at the next step is $0$. Combining the above inequality with the Markov property leads to the desired result. 
\end{proof}

Turn next to an explanation of the idea behind the coupling. The argument uses path coupling -- that is, coupling a pair of split-merge random walks started at a neighboring pair of elements. This technique was invented by Bubley and Dyer in \cite{BubleyDyer}; a good reference is Chapter 14 of \cite{YuvalBook}. 
To be precise, endow the state space $\Omega$ with a connected graph structure: that is, select a set of edges $E'$ between elements of $\Omega$, such that for any $u, v \in \Omega$, there exists a path between $u$ and $v$ only using the edges in $E'$. It is then only necessary to define a coupling for $(x, y) \in E'$. 

Assign lengths $l(x, y)\geq 1$ to each edge $(x,y) \in E'$, and define a path metric $\rho$ on $\Omega$ by 
\begin{align*}
\rho(z, w) = \min \left\{\sum_{i = 0}^{n-1} l(x_i, x_{i+1}) \left| \right. x_0 = z, x_n = w, (x_i, x_{i+1}) \in E' \text{ for all } i \right\}
\end{align*}
Furthermore, define the diameter of the set $\Omega$ in the usual way as $\mathrm{diam}(\Omega) = \max_{u, v \in \Omega} \rho(u, v)$
The following theorem is the basic path coupling bound. 
\begin{theorem}\label{pathcoupling} 
Let $(X_t)_{t\geq 0}$ be a Markov chain on a set $\Omega$, and let $E'$, $l$ and $\rho$ be defined as above. Let $(X_1, Y_1)$ be the first step of a coupling started at $(x, y) \in E'$. Then, if there is a $\kappa <1$ such that for every $(x, y) \in E'$, 
\begin{equation}\label{condforpathcoupling}
\mathbb{E}\left[ \rho(X_1, Y_1)\right] \leq \kappa \rho(x, y)
\end{equation}
then for all $t \geq 1$, 
\begin{equation*}
d(t) \leq \textnormal{diam}(S) \kappa^t
\end{equation*} 
\end{theorem}

Returning to the random walk under consideration, define neighboring pairs of partitions to be precisely the pairs which are one step away in the split-merge random walk. Then, define a coupled process $(X_t, Y_t)$ such that $X_0 = \sigma$ and $Y_0 = \tau$, making sure that the distance between $X_t$ and $Y_t$ at each step is at most $1$. Here are some useful definitions. 

\begin{defn}\label{rhodefn}
For $\sigma$ and $\tau$ partitions of $n$, define $\rho(\sigma, \tau)$ to be the distance between $\sigma$ and $\tau$ induced by the split-merge random walk; that is, $\rho(\sigma, \tau)$ is the number of split-merge steps it takes to get from $\sigma$ to $\tau$. 
\end{defn}

The next definition is useful for finding a lower bound on the probability of coupling at each step given the current location of the two walks. 

\begin{defn}\label{smdefn} 
Let $\sigma$ and $\tau$ be partitions of $n$ such that $\rho(\sigma, \tau) = 1$. Then $\sigma$ and $\tau$ are exactly one merge away, so rearranging parts appropriately and without loss of generality letting $\sigma$ be the partition with more parts, 
\begin{equation}\label{sigmataucanonical}
\begin{split} 
\sigma &= (a_1, a_2, \dots, a_m, b, c) \\
\tau  &= (a_1, a_2, \dots, a_m, b+ c) 
\end{split} 
\end{equation}
where $b \leq c$. Then, define 
\begin{equation}\label{smequation}
s(\tau, \sigma)= s(\sigma, \tau) = b \ \text{ and }\ m(\tau, \sigma) =  m(\sigma, \tau) = c
\end{equation}
That is, since $\sigma$ and $\tau$ differ in the parts $b, c$ and $b+c$, $s(\sigma, \tau)$ is the smallest part in which they differ, and $m(\sigma, \tau)$ is the medium part in which they differ. 

For later use, define $m(\sigma, \sigma) = n$ and $s(\sigma, \sigma)= \frac{n}{2}$. 

\end{defn}

In the next section, the coupling is given along with the following lemma: 

\begin{lemma}
\label{schrammcouplinglowerbound}
Assume that $(X_t, Y_t)= (\sigma, \tau)$, for $\sigma$ and $\tau$ such that $\rho(\sigma, \tau) = 1$. Then, $\rho(X_{t+1}, Y_{t+1}) \leq 1$, and 
\begin{equation*}
\prob (X_{t+1} = Y_{t+1}) \geq \frac{4 s(\sigma, \tau)}{n^2}
\end{equation*}
That is, the chain stays at most distance $1$ apart, and gives the above lower bound for the probability of coupling. 
\end{lemma}

After proving the above lemma, it is shown below that after order $n$ steps, $s(X_t, Y_t)$ is on average of order $n$. The lemma then implies that the probability of coupling at each step is of order $\frac{1}{n}$, which will show that there is a high probability of coupling after order $n$ steps. Using the fact that the diameter of the set of partitions is no greater than $n$, Theorem \ref{pathcoupling} shows that the random transposition walk mixes in $O(n \log n)$ time.

\section{The Coupling} 
This section defines the coupling for neighboring pairs for the split-merge random walk, and proves Lemma \ref{schrammcouplinglowerbound}. The coupling is defined in such a way that the distance between $X_t$ and $Y_t$ at each step is at most $1$ for all $t$. As usual, once the two chains meet, they are run together. 

\begin{defn}\label{schrammcouplingdefn}
Consider the next step $(X_1, Y_1)$ of a coupling which is currently at $(X_0, Y_0) = (\sigma, \tau)$, where $\rho(\sigma, \tau) = 1$ and 
\begin{align*}
\sigma &= (a_1, a_2, \dots, a_m, b, c)\\
\tau &= (a_1, a_2, \dots, a_m, b+c)
\end{align*}
 where $b\leq c$. There are a number of cases, considered in the following order: go through the possible moves in $\sigma$, then provide corresponding moves in $\tau$. 

\begin{itemize}
\item {\bf Operations only using the $a_i$:} If $a_i$ and $a_j$ are merged in $\sigma$ for any $i$ and $j$, perform the same operation in $\tau$. Similarly, if $a_i$ is split in $\sigma$ into $\{r, a_i-r\}$, do the same for $a_i$ in $\tau$. Then, 
\begin{align*}
X_1 &= (a_1', \dots, a_k', b, c)\\
Y_1 &= (a_1', \dots,  a_k', b+c) 
\end{align*}
for the appropriate $\{a_1', a_2', \dots, a_k'\}$.

\item {\bf Merging $b$ or $c$ and $a_i$:} If $b$ and $a_i$ are merged in $\sigma$, merge $b+c$ and $a_i$ in $\tau$. If $c$ and $a_i$ are merged in $\sigma$, also merge $b+c$ and $a_i$ in $\tau$. In the first case,  
\begin{align*}
X_1 &= (a_1', \dots, a_{m-1}', b+a_i, c)\\
Y_1 &= (a_1', \dots,  a_{m-1}', b+c+a_i) 
\end{align*}
where  $\{a_1', a_2', \dots, a_{m-1}'\} = \{a_1, a_2, \dots, a_m\}/\{a_i\}$. The case where $c$ and $a_i$ are merged in $\sigma$ is analogous. 

\item {\bf Splitting $b$ or $c$:} If $b$ is split in $\sigma$ into $\{r, b-r\}$ where $r \leq \frac{b}{2}$, then split $b+c$ in $\tau$ into $\{r, b+c-r\}$. Similarly, if $c$ is split in $\sigma$ into $\{r, c-r\}$ where $r \leq \frac{c}{2}$, then split $b+c$ in $\tau$ into $\{r, b+c-r\}$. The first case results in 
\begin{align*}
X_1 &= (a_1, \dots, a_m, r, b-r, c)\\
Y_1 &= (a_1, \dots,  a_m,r, b+c-r) 
\end{align*}
The second case, where $c$ is split into $\{r, c-r\}$, is analogous. 

\item {\bf Staying in place:} If the walk stays in place in $\sigma$, it is coupled with either staying in place in $\tau$ or with splitting $b+c$ in $\tau$ into $\{b, c\}$. Since splitting $b+c$ into $\{b, c\}$ may have already been coupled with splitting $c$ into $\{b, c-b\}$, let $p$ be the remaining probability of splitting $b+c$ into $\{b, c\}$. Then, couple staying in place in $\sigma$ with splitting $b+c$ into $\{b, c\}$ in $\tau$ with probability $\min\left(p, \frac{1}{n}\right)$. This results in
\begin{align*}
X_1 &= (a_1, \dots, a_m,  b, c)\\
Y_1 &= (a_1, \dots,  a_m, b, c) 
\end{align*}
That is, the chains will couple. 

Couple staying in place in $\sigma$ to staying in place in $\tau$ the rest of the time -- that is, with probability $\frac{1}{n} - \min\left(p, \frac{1}{n}\right)$. 

\item {\bf Merging $b$ and $c$:} Couple merging $b$ and $c$ in $\sigma$ to any remaining possibilities in $\tau$. It is easy to check that these are either staying in place or splitting $b+c$ into $\{r, b+c-r\}$. The first leads to the chains coupling; the second leads to 
\begin{align*}
X_1 &= (a_1, \dots, a_m,  b+ c)\\
Y_1 &= (a_1, \dots,  a_m, r, b+ c-r) 
\end{align*}
for some $r$. 

\end{itemize}
\end{defn}  

\begin{example}
As this coupling looks fairly complicated, here are a couple of examples. The possible pairs for $(X_1, Y_1)$ are listed, as well as the probability of each pair. 
 
\begin{enumerate}
\item 
Let $(X_0, Y_0) = (\sigma, \tau) = ((2,3),(5))$. Here, there are no $a_i$, $b= 2$, $c=3$, and $b+c = 5$. A description is provided for each pair of moves: the first move corresponds to $\sigma$, the second to $\tau$. 
\begin{equation*}
(X_1, Y_1) :
\begin{cases} 
((1,1,3),(1,4)), p = \frac{2}{25} &\text{split $2$ as $\{1,1\}$, split $5$ as $\{1,4\}$}\\
((1,2,2),(1,4)), p = \frac{6}{25} &\text{split $3$ as $\{1,2\}$, split $5$ as $\{1, 4\}$}\\
((2,3), (2,3)), p = \frac{5}{25} & \text{stay at $\sigma$, split $5$ as $\{2, 3\}$} \\
((5),(1,4)), p = \frac{2}{25} &\text{merge $2$ and $3$, split $5$ as $\{1,4\}$}\\
((5), (2,3)), p = \frac{5}{25} &\text{merge $2$ and $3$, split $5$ as $\{2,3\}$}\\
((5), (5)), p = \frac{5}{25} &\text{merge  $2$ and $3$, stay at $\tau$}
\end{cases}
\end{equation*}

\item Let $(X_0, Y_0) = (\sigma, \tau) = ((2, 1,3), (2,4))$, written with the above convention that the parts $\sigma$ and $\tau$ disagree on are written last. Here, $a_1 = 2$, $b=1$, $c=3$, and $b+c  = 4$, and the first move again corresponds to $\sigma$, while the second corresponds to $\tau$. 
\begin{equation*}
(X_1, Y_1) :
\begin{cases} 
((1,1,1,3),(1,1,4)), p = \frac{2}{36} &\text{split $2$ as $\{1, 1\}$ in both}\\
((3,3),(6)), p = \frac{4}{36} &\text{merge $2$  and $1$, merge $2$ and $4$}\\
((1,5), (6)), p = \frac{12}{36} & \text{merge $2$ and $3$, merge $2$ and $4$}\\
((2,1,1,2),(2,1,3)), p = \frac{6}{36} &\text{split $3$ as $\{1, 2\}$, $4$ as $\{1, 3\}$}\\
((2,1,3),(2,1,3)), p = \frac{2}{36} &\text{stay at $\sigma$, split $4$ as $\{1, 3\}$}\\
((2,4),(2,2,2)), p = \frac{4}{36} &\text{merge $1$ and $3$, split $4$ as $\{2, 2\}$}\\
((2,4),(2,4)), p = \frac{2}{36} &\text{merge $1$ and $3$, stay at $\tau$}\\
((2,1,3),(2,4)), p = \frac{4}{36} &\text{stay at $\sigma$ and $\tau$}\\
\end{cases}
\end{equation*}
\end{enumerate}

\end{example}

Going back to the general case, here is a check that the above definition provides the correct marginal distribution for $Y_1$. Note that given the way that the coupling was defined, it clearly provides the correct distribution for $X_1$. 

\begin{lemma}
The coupling in Definition \ref{schrammcouplingdefn} has the correct marginal distribution for $Y_1$. 
\end{lemma} 

\begin{proof}[\bf Proof:]

Since $\sigma$ and $\tau$ share the parts $a_i$, the operations only using the $a_i$ are distributed identically in both and hence pose no problem. Furthermore, 
\begin{equation*}
\begin{split}
\prob(\text{Merge $b$ and $a_i$ in $\sigma$}) + \prob(\text{Merge $c$ and $a_i$ in $\sigma$}) &= \frac{2ba_i}{n^2} + \frac{2ca_i}{n^2} = \frac{2(b+c)a_i}{n^2} \\
                    &= \prob(\text{Merge $b+c$ and $a_i$ in $\tau$})
\end{split}
\end{equation*} 
Thus, all the operations involving any $a_i$ work properly. 

Consider next operations that only involve $b$ and $c$ in $\sigma$. Splitting $b$ into $\{r, b-r\}$ where $r\leq \frac{b}{2}$ in $\sigma$ is coupled with splitting $b+c$ into $\{r, b+c-r\}$ in $\tau$, and similarly for $c$. It needs to be checked that this is possible -- that is, that the probability of splitting $b+c$ into $\{r, b+c-r\}$ in $\tau$ is sufficiently large to accommodate all these moves in $\sigma$. 

There are a number of possibilities. First of all, if $r \leq \frac{b}{2}$, then clearly $r < \frac{b+c}{2}$, and hence according to Definition \ref{splitmergedesc}, 
\begin{align*}
\prob(\text{Split $b+c$ into $\{r, b+c-r\}$ in $\tau$}) &= \frac{2(b+c)}{n^2} = \frac{2b}{n^2} + \frac{2c}{n^2} \\ 
&\geq \prob(\text{Split $b$ into $\{r, b-r\}$ in $\sigma$})  +\\
&\hspace{30 pt}  \prob(\text{Split $c$ into $\{r, c-r\}$ in $\sigma$}) 
\end{align*} 
In this case, the probability of splitting $b+c$ into $\{r, b+c-r\}$ in $\tau$ is sufficiently large. 

x

Now, if $\frac{b}{2} < r \leq \frac{c}{2}$, the procedure couples splitting $b+c$ into $\{r, b+c-r\}$ with splitting $c$ into $\{r, c-r\}$. Thus, since in this case $r$ is still less than  $\frac{b+c}{2}$,   
\begin{align*}
\prob(\text{Split $b+c$ into $\{r, b+c-r\}$ in $\tau$}) &= \frac{2(b+c)}{n^2} \geq \frac{2c}{n^2} \\ 
&\geq \prob(\text{Split $c$ into $\{r, c-r\}$ in $\sigma$}) 
\end{align*} 
which again works. 

Finally, if $r > \frac{c}{2}$, splitting $b+c$ into $\{r, b+c-r\}$ is not coupled to splitting either $b$ or $c$ in $\sigma$, which obviously does not pose a problem. None of the other moves considered in Definition \ref{schrammcouplingdefn} could be an issue, and hence the marginal distribution of $Y_1$ under this definition is correct. 
\end{proof}

The next step proves Lemma \ref{schrammcouplinglowerbound}. This states that the coupled chains stay at most one step apart, and that 
\begin{equation*}
\prob (X_{t+1} = Y_{t+1}) \geq \frac{4 s(X_t, Y_t)}{n^2}
\end{equation*}

\begin{proof}[\bf Proof of Lemma \ref{schrammcouplinglowerbound}:] 
It should be clear from Definition \ref{schrammcouplingdefn} that the coupling stays at most one step apart for all $t$. To show that if $(X_t, Y_t) = (\sigma, \tau)$, where $\rho(\sigma, \tau) = 1$, then 
\begin{equation*}
\prob (X_{t+1} = Y_{t+1}) \geq \frac{4 s(\sigma, \tau)}{n^2}
\end{equation*} 
let 
\begin{align*}
\sigma &= (a_1, \dots, a_m, b, c) \\
\tau &= (a_1, \dots, a_m, b+c) 
\end{align*} 
where $b \leq c$. Then by Definition \ref{smdefn}, $s(\sigma, \tau) = b$. 

From Definition \ref{schrammcouplingdefn}, the chains can couple either if $\sigma$ stays in place, or if $b$ and $c$ are merged in $\sigma$. Consider those two cases separately. 

\paragraph {\bf Staying in place in $\sigma$:} 
The chains will couple if $\sigma$ stays in place and $b+c$ is split in $\tau$ into $\{b, c\}$. As noted in the definition, these are coupled together with probability $\min\left(p, \frac{1}{n}\right)$, where $p$ is the remaining probability of splitting $b+c$ into $\{b, c\}$ in $\tau$ -- the probability that this split hasn't already been coupled to something else. To find a lower bound on $p$, first note that splitting $b+c$ into $\{b, c\}$ in $\tau$ couldn't have been coupled with any splits of $b$ in $\sigma$. However, it might have been coupled with a split of $c$ in $\sigma$. Consider two cases: $c < 2b$ and $c \geq 2b$. 

If $c < 2b$, then splitting $c$ into $\{c-b, b\}$ in $\sigma$ is coupled to splitting $b+c$ into $\{c-b, 2b\}$ in $\tau$ since $c - b < b$. This means that nothing is coupled to splitting $b+c$ into $\{b, c\}$, and therefore 
\begin{equation}\label{csmall} 
p = \prob(\text{Splitting $b+c$ into $\{b, c\}$ in $\tau$}) \geq \frac{b+c}{n^2} \geq \frac{2b}{n^2}
\end{equation}

If $c \geq 2b$, then splitting $b+c$ into $\{b, c\}$ in $\tau$ is indeed coupled with splitting $c$ into $\{b, c-b\}$ in $\sigma$. In this case, clearly $b \neq c$, and hence  
\begin{equation*} 
\prob(\text{Splitting $b+c$ into $\{b, c\}$ in $\tau$}) = \frac{2(b+c)}{n^2} 
\end{equation*}
Therefore, 
\begin{align}\label{clarge}
p &\geq \prob(\text{Splitting $b+c$ into $\{b, c\}$ in $\tau$}) - \prob(\text{Splitting $c$ into $\{b, c-b\}$ in $\sigma$})\nonumber \\
  & \geq  \frac{2(b+c)}{n^2} - \frac{2c}{n^2} = \frac{2b}{n^2} 
\end{align} 
Equations \eqref{csmall} and \eqref{clarge} give $p \geq \frac{2b}{n^2}$. Furthermore, note that $b\leq c$, and $b+c \leq n$, and hence  $b \leq \frac{n}{2}$. Therefore, 
\begin{equation}\label{minplowerbound}
\min\left(p, \frac{1}{n} \right) \geq \min\left(\frac{2b}{n^2}, \frac{1}{n} \right) \geq \frac{2b}{n^2} 
\end{equation} 
Hence, 
\begin{equation}\label{lowerboundstayinplace}
\prob(\text{Coupling if staying in place in $\sigma$}) = \min\left( p, \frac{1}{n}\right) \geq \frac{2b}{n^2} 
\end{equation}

\paragraph {\bf Merging $b$ and $c$ in $\sigma$:} Next, consider the probability of coupling if $b$ and $c$ are merged in $\sigma$. Clearly, this would need to be coupled with staying in place in $\tau$. The only other thing that staying in place in $\tau$ could have been coupled with so far is staying in place in $\sigma$. As noted in Definition \ref{schrammcouplingdefn}, 
\begin{equation*}
\prob(\text{Both $\sigma$ and $\tau$ stay in place}) = \frac{1}{n}- \min\left(p, \frac{1}{n}\right)
\end{equation*}
for the same $p$ used above. Thus,  
\begin{equation}\label{lowerboundmergingbandc}
\begin{split}
\prob(\text{Coupling if merging $b$ and $c$ in $\sigma$}) &= \prob(\text{$b$ and $c$ merged in $\sigma$, $\tau$ stayed}) \\
 &= \frac{1}{n} - \prob(\text{Both $\sigma$ and $\tau$ stayed}) \\
&= \min\left(p, \frac{1}{n}\right) \geq \frac{2b}{n^2}
\end{split}
\end{equation}
using Equation \eqref{minplowerbound} above. 

Finally, combining Equations \eqref{lowerboundstayinplace} and \eqref{lowerboundmergingbandc},
\begin{equation*}
\prob(X_{t+1} = Y_{t+1}) \geq \frac{4b}{n^2} = \frac{4s(\sigma, \tau)}{n^2} = \frac{4s(X_t, Y_t)}{n^2}
\end{equation*} 
as required. 
\end{proof}

Continuing with the proof, as sketched out earlier, the rest of this paper will be concerned with showing that $s(X_t, Y_t)$ is of order $n$ after $O(n)$ time. The next section shows how that proves Theorem \ref{schrammcouplingtheorem}, and provides a summary of the proof. 

\section{Proof of Main Theorem Using $\expect \left[ s(X_t, Y_t) \right]$ } 

As described above, one of the main tools of this paper is the following theorem: 

\begin{theorem} \label{schrammsislargeonaverage}
There exist constants $\alpha$ and $\beta$ such that for all $t \geq \alpha n$, 
\begin{equation*}
\expect \left[s(X_t, Y_t)\right] \geq \beta n 
\end{equation*}
\end{theorem} 
This section uses the above result to prove Theorem \ref{schrammcouplingtheorem}. To start, prove the following easy lemma:

\begin{lemma} \label{expectationcouplinginequality}
Let $(X_t, Y_t)$ be defined as in Definition \ref{schrammcouplingdefn}, where as usual $\rho(X_0, Y_0)$ is equal to $1$. Let $\alpha$ and $\beta$ be the constants in Theorem \ref{schrammsislargeonaverage} above. Then, 
\begin{equation*}
\prob\left(X_{\alpha n   + \frac{n}{2}} = Y_{ \alpha n  + \frac{n}{2}} \right) \geq \beta
\end{equation*}
\end{lemma}
\begin{proof}[\bf Proof:]
Since by Lemma \ref{schrammcouplinglowerbound}, $\prob(X_t  = Y_t)$ is non-decreasing, if $\prob(X_t = Y_t) \geq \beta$ for any $t \leq \alpha n + \frac{n}{2}$, the argument is complete. Thus, assume that 
\begin{equation}\label{assumptionplowerbound} 
\prob(X_t = Y_t) \leq \beta
\end{equation}
for all $t \leq \alpha n + \frac{n}{2}$. 

Clearly, 
\begin{equation*}
\prob(X_{t+1} = Y_{t+1}) = \prob(X_t = Y_t) + \prob(X_{t+1} = Y_{t+1} \left| \right. X_t \neq Y_t) \prob(X_t \neq Y_t)
\end{equation*}
Rearranging, and using Lemma \ref{schrammcouplinglowerbound}, 
\begin{equation}\label{increaselowerbound} 
\begin{split}
\prob(X_{t+1} = Y_{t+1}) - \prob(X_t = Y_t) &\geq  \expect\left[ \frac{4s(X_t, Y_t)}{n^2} \Big| X_t \neq Y_t\right] \prob(X_t \neq Y_t)\nonumber \\
&= \frac{4}{n^2}\expect\left[ s(X_t, Y_t) \Big| X_t \neq Y_t\right] \prob(X_t \neq Y_t)
\end{split}
\end{equation}
A lower bound is now needed for the right-hand side. Assume that $t\geq \alpha n$, and hence that $\expect\left[s(X_t, Y_t) \right] \geq \beta n$ by Lemma \ref{schrammsislargeonaverage}. Then, 
\begin{align*}
\expect\left[ s(X_t, Y_t) \Big| X_t \neq Y_t\right] \prob(X_t \neq Y_t) &= \expect\left[ s(X_t, Y_t)\right] - \\
& \hspace{30 pt} \expect \left[ s(X_t, Y_t) \Big| X_t = Y_t \right] \prob(X_t = Y_t) \\
&\geq \beta n - \frac{n}{2} \prob \left(X_t = Y_t \right)
\end{align*}
since if $X_t = Y_t$, $s(X_t, Y_t) = \frac{n}{2}$. Furthermore, using Equation \eqref{assumptionplowerbound}, 
\begin{equation*}
\expect\left[ s(X_t, Y_t) \Big| X_t \neq Y_t\right] \prob(X_t \neq Y_t) \geq \frac{\beta n }{2} 
\end{equation*}
Combining this with Equation \eqref{increaselowerbound}, 
\begin{equation*}
\prob(X_{t+1} = Y_{t+1}) - \prob(X_t = Y_t) \geq \frac{ 2\beta  }{n}
\end{equation*} 
for all $\alpha n \leq t \leq \alpha n + \frac{n}{2}$. Adding up these inequalities for all $t$ in $[\alpha n , \alpha n + \frac{n}{2} ]$,
\begin{equation*}
\prob\left(X_{\alpha n + \frac{n}{2}} = Y_{\alpha n  +\frac{n}{2}}\right) \geq \beta
\end{equation*} 
as required. 
\end{proof} 

For path coupling, a lemma about the diameter of $P_n$ under the split-merge random walk is needed.  

\begin{lemma}\label{diametersetofpartitions} 
The diameter of $P_n$ under the split-merge random walk is at most $n$. 
\end{lemma}
\begin{proof}[\bf Proof:]
Proceed by induction on $n$. This statement is clearly true for $n= 1$. Now, assume it's true for all $m \leq n-1$, and show it for $n$. Let $\sigma = (a_1, \dots, a_k)$ and $\tau = (b_1, \dots, b_l)$ be two partitions of $n$. Without loss of generality, assume that $a_1 \geq b_1$. 

If $a_1 = b_1$, create a path from $\sigma$ to $\tau$ by just changing the parts $(a_2, \dots, a_k)$ to $(b_2, \dots, b_l)$. Since $(a_2, \dots, a_k)$ is a partition of $n - a_1$, by the inductive hypothesis, 
\begin{equation*}
\rho(\sigma, \tau) \leq n - a_1 \leq n -1. 
\end{equation*} 
so this case follows. 

Otherwise, $a_1 > b_1$. Let $\sigma_1$ be $\sigma$ with $a_1$ split into $(b_1, a_1 - b_1)$. Then, $\sigma_1$ and $\tau$ match on the part $b_1$, and hence by the argument above, 
\begin{equation*}
\rho(\sigma_1, \tau) \leq n - 1
\end{equation*}
Since $\sigma$ is a neighbor of $\sigma_1$, this implies that $\rho(\sigma, \tau) \leq n$, completing the proof. 
\end{proof}

Theorem \ref{schrammcouplingtheorem} is now proved using path coupling. It shows an $O(n \log n)$ bound on the split-merge random walk, and hence on the random transposition walk. 

\begin{proof} [\bf Proof of Theorem  \ref{schrammcouplingtheorem}]
Let $t_1 = \alpha n + \frac{n}{2}$. Consider the walk $(\tilde{X}_k)_{k\geq 1}$, where each step consists of making $t_1$ steps of the split-merge random walk. Let $(\tilde{X}_k, \tilde{Y}_k)$ be the coupling on this new walk induced by the current coupling $(X_t, Y_t)$. Now, Proposition \ref{expectationcouplinginequality} shows that if $(\tilde{X}_0, \tilde{Y}_0) = (\sigma, \tau)$, where $\rho(\sigma, \tau) = 1$, then 
\begin{align*}
\expect \left[ \rho(\tilde{X}_1, \tilde{Y}_1)\right]  &= \expect \left[ \rho(X_{t_1}, Y_{t_1}) \right] = \prob(X_{t_1} \neq Y_{t_1}) \\
&\leq (1 - \beta) \rho(\sigma, \tau)
\end{align*}
using the fact that $\rho(X_t, Y_t)$ is always either $0$ or $1$. 
Therefore, if $\tilde{d}(k)$ is defined to be the distance from stationarity of $(\tilde{X}_k, \tilde{Y}_k)$, then from Theorem \ref{pathcoupling},  
\begin{equation*}
\tilde{d}(k) \leq \textnormal{diam}(P_n) \left( 1 - \beta \right)^k 
\end{equation*}
Since neighboring pairs are pairs that are one step apart in the split-merge random walk, Proposition \ref{diametersetofpartitions} implies that $\textnormal{diam}(P_n) \leq n$. Also using the fact that $1 - x \leq e^{-x}$, 
\begin{equation*}
\tilde{d}(k)  \leq n e^{-\beta k} 
\end{equation*}
Thus, if $k = \frac{\log n }{2\beta}$, then $\tilde{d}(k) \leq e^{-2} <\frac{1}{4}$. But it's clear from the definition of the new walk that 
\begin{equation*}
d(kt_1) = \tilde{d}(k) 
\end{equation*}
Thus, 
\begin{equation*}
d\left(\left( \frac{\alpha}{2\beta} + \frac{1}{4\beta}\right) n \log n \right)  = d(kt_1) < \frac{1}{4}
\end{equation*} 
which means that the walk has mixed by time $\left( \frac{\alpha}{2\beta} + \frac{1}{4\beta}\right) n \log n$, completing the proof.  
\end{proof}

\section{Proving $\expect \left[ s(X_t, Y_t) \right]$ is large}
Let us now summarize the rest of the proof. The remainder of this paper will be devoted to proving Theorem \ref{schrammsislargeonaverage}, which states that after $O(n)$ time, the expected value of $s(X_t, Y_t)$ is of order $n$. 

The proof will be structured as follows: it is shown that in $O(n)$ time, $s(X_t, Y_t)$ will have a high probability of being at least order $n^{1/3}$. Then it is shown that it takes another $o(n)$ time for $s(X_t, Y_t)$ to have a high probability of being of order $n$. This will clearly suffice to show that that after $O(n)$ time, $\expect \left[ s(X_t, Y_t) \right]$ is of order $n$. Section \ref{orderrootnsection} below will be concerned with growing $s(X_t, Y_t)$ to order $n^{1/3}$, while Section \ref{ordernsection} will be concerned with growing it to order $n$. 

Before stating the theorems and sketching their proofs, a number of useful definitions are needed. Note that some of these definitions are asymmetrical: they are defined in terms of $\bar{X}_t$ and not $\bar{Y}_t$. This is an arbitrary choice; since the pair $(X_t, Y_t)$ is only a step apart, it doesn't make any difference.

\begin{defn} \label{Vtdefn}
For $v \in \{1, 2, \dots, n\}$, define $C_t(v)$ to be the cycle of $\bar{X}_t$ containing $v$. Furthermore, for a number $x$, define 
\begin{equation*}
V_t(x) = \{ v \in \{1, 2, \dots, n\} \left| \right. \left|C_t(v) \right| \geq x \}
\end{equation*} 
Thus, $V_t(x)$ is the union of all cycles of size at least $x$. 
\end{defn}
\begin{rmk} \label{Vtremark} 
Note that if $X_t = (a_1, a_2, \dots, a_m)$, then
\begin{equation*}
\left| V_t(x) \right| = \sum_{a_i \geq x} a_i
\end{equation*} 
Thus, the size of $V_t(x)$ is a function of $X_t$. 
\end{rmk}

The first proposition that grows $s(X_t, Y_t)$ to order $n^{1/3}$ is now stated. 

\begin{prop}\label{rootnlemma}
Let $(X_t, Y_t)$ be the usual coupling started at $(X_0, Y_0) = (\sigma, \tau)$, where $\rho(\sigma, \tau) \leq 1$. Then, for $n$ sufficiently large and $t\geq 9n$, 
\begin{equation*}
\prob \left\{s(X_t, Y_t) \geq n^{1/3}, \left| V_t\left(n^{1/3}\right)\right| \geq \frac{n}{2}\right\} \geq \frac{1}{2}
\end{equation*} 
\end{prop}
\begin{rmk}
Here, the choice of $n^{1/3}$ is in some sense arbitrary -- any $n^{\alpha}$, where $\alpha < \frac{1}{2}$, would have done just as well. 
\end{rmk}

A few other definitions which are needed for the statement of the theorem growing $s(X_t, Y_t)$ from order $n^{1/3}$ to order $n$. Indeed, a more general theorem is proved. Fix constants $\epsilon$ and $\delta$: then, if $s(X_t, Y_t)$ starts by being of size $2^{j+1}$ (where $j$ can be a function of $n$), after a certain amount of time $q$, $s(X_{t+q}, Y_{t+q})$ has a high probability of being least $\epsilon \delta n$. The following definition introduces some notation necessary for stating the theorem; it currently looks completely inexplicable, but will be justified in Section \ref{ordernsection}. 

\begin{defn}\label{Kandtaudef}
Assume $\epsilon$ and $\delta$ are fixed constants, and $j$ is a number (possibly a function of $n$). Then, define 
\begin{equation}\label{Kdef}
K = \lceil \log_2(\epsilon \delta n) \rceil 
\end{equation}
Furthermore, for $r$ between $j$ and $K$ define 
\begin{equation}\label{taudef}
a_r = \lceil 2 \delta^{-1} 2^{-r} n (\log_2 n - r)\rceil \ \text{ and }\ 
\tau_r = \sum_{i=j}^{r-1} a_i
\end{equation}
where as usual, $\lceil \cdot \rceil$ stands for the ceiling function. 
\end{defn}

The following proposition proves that $s(X_t, Y_t)$ grows to order $n$. 

\begin{prop}\label{sbiglemma}
Let $(X_t, Y_t)$ be the usual coupling started at $(X_0, Y_0) = (\sigma, \tau)$, where $\rho(\sigma, \tau) \leq 1$. Let $j$ be a number and let $\delta \in (0,1]$ be a constant such that $\left|V_0(2^{j+1})\right| \geq \delta n$ and $s(\sigma, \tau) \geq 2^{j+1}$. If $K$ and $\tau_K$ are defined as in Definition \ref{Kandtaudef} and $\epsilon \in (0, 1/32)$, then
\begin{equation}\label{provesgrow} 
\prob\{s(X_{\tau_K}, Y_{\tau_K}) < \epsilon \delta n\} \leq O(1) \delta^{-1} \epsilon \left| \log(\epsilon \delta)\right| 
\end{equation} 
where the constant implied in the $O(1)$ notation is universal.
\end{prop}

\begin{proof}[\bf Proof of Theorem \ref{schrammsislargeonaverage}]
Propositions \ref{rootnlemma} and \ref{sbiglemma} can be used to prove Theorem \ref{schrammsislargeonaverage}: let $t_1 \geq 9n$, and condition on $(X_{t_1}, Y_{t_1})\in Q_{t_1}$, where
\begin{equation}\label{Pdef}
Q_{t_1}=\left\{(X_{t_1}, Y_{t_1}) \text{ such that } s(X_{t_1}, Y_{t_1}) \geq n^{1/3}, \left| V_{t_1}^\pi\left(n^{1/3}\right)\right| \geq \frac{n}{2}\right\} 
\end{equation}
Letting $2^{j+1} = n^{1/3}$ and $\delta = \frac{1}{2}$, if $(X_{t_1},Y_{t_1}) \in Q_{t_1}$, then 
\begin{equation*}
s(X_{t_1},Y_{t_1}) \geq 2^{j+1} \text{ and } \left|V_{t_1}(2^{j+1})\right| \geq \delta n 
\end{equation*}
Since $\rho(X_{t_1}, Y_{t_1}) \leq 1$, Proposition \ref{sbiglemma} applies to pairs $(X_{t_1}, Y_{t_1})$ in $Q_{t_1}$. Therefore, averaging over $(X_{t_1},Y_{t_1}) \in Q_{t_1}$, 
\begin{equation*}
\prob\left\{s(X_{t_1+\tau_K}, Y_{t_1+\tau_K}) < \epsilon \delta n \mid (X_{t_1}, Y_{t_1}) \in Q_{t_1} \right\} \leq O(1) \delta^{-1} \epsilon \left| \log(\epsilon \delta)\right|
\end{equation*} 
for any $\epsilon \in (0, 1/32)$. Now, pick $\epsilon$ such that the right hand side of the above inequality is at most $1/2$. Then, 
\begin{equation*}
\prob\left\{s(X_{t_1+\tau_K}, Y_{t_1+\tau_K}) \geq  \epsilon \delta n \mid (X_{t_1}, Y_{t_1}) \in Q_{t_1} \right\} \geq \frac{1}{2}
\end{equation*}\
and therefore, for sufficiently large $n$,
\begin{equation*}
\prob \left\{s(X_{t_1+\tau_K}, Y_{t_1+\tau_K}) \geq  \epsilon \delta n  \right\} \geq \frac{\prob (Q_{t_1})}{2} \geq \frac{1}{4}
\end{equation*}
using Lemma \ref{rootnlemma}. Therefore, 
\begin{equation}\label{almostthere} 
\expect \left[ s(X_{t_1+ \tau_K}, Y_{t_1 + \tau_K}) \right] \geq \frac{\epsilon \delta n}{4}
\end{equation}
It now just remains to show that is that $t_1 + \tau_K$ can be of order $n$. Since $\delta = \frac{1}{2}$ and $2^{j+1} = n^{1/3}$, by Equation \eqref{taudef}
\begin{align*}
\tau_K &= \sum_{i=j}^{K-1} \lceil 2 \delta^{-1} 2^{-i} n (\log_2 n - i)\rceil  = O\left( n \log n \sum_{r=j}^{K-1} 2^{-i} \right) \\
&= O\left( n \log n\cdot 2^{-j+1}  \right)
 = O( n^{2/3} \log n)
\end{align*}
Since $t_1\geq 9n$ is arbitrary and $\tau_K$ is $o(n)$, Equation \eqref{almostthere} implies that 
\begin{equation*}
\expect \left[ s(X_t, Y_t) \right] \geq \frac{\epsilon \delta n}{4}
\end{equation*}
for all $t \geq 10n$, which is precisely what is needed.
\end{proof}

Before the next two sections, in which Propositions \ref{rootnlemma} and \ref{sbiglemma} are proved, some technical results are needed. These are proved in Section \ref{schrammtechnicallemmas} below, and are instrumental for controlling the probabilities in the next two sections. 

\begin{lemma}\label{schrammshrinklemma}
Let $\sigma$ be in $S_n$, and let $(\bar{X}_t)_{t\geq 1}$ be the random transposition walk starting at $\sigma$. Then, the expected number of $v$ such that $\left|C_1(v)\right| < \left|C_0(v)\right|$ and $\left|C_1(v)\right| < x$ is no greater than $\frac{x^2}{n}$. 
\end{lemma}

For the next four lemmas, let $(X_t, Y_t)$ be the usual coupling starting at $(\sigma, \tau)$, where $\rho(\sigma, \tau) = 1$, $s(\sigma, \tau) = b$ and $m(\sigma, \tau) = c$. 

\begin{lemma}\label{mshrinklemma}
If $x \leq c$, then 
\begin{equation*}
\prob\left\{m(X_1, Y_1) < x\right\} \leq \frac{2x^2}{n^2}. 
\end{equation*}
\end{lemma}

\begin{lemma}\label{mgrowthlemma}
If $x \leq c$, and if $| V_0(y)| \geq R$, then
\begin{equation*}
\prob \{m(X_1, Y_1) \geq x +y \} \geq  \frac{2c(R- 2c)}{n^2} 
\end{equation*}
\end{lemma}

\begin{lemma}\label{sshrinklemma}
If $x \leq b$, then 
\begin{equation*}
\prob \{s(X_1, Y_1) < x\} \leq \frac{4x^2}{n^2}
\end{equation*}
\end{lemma}

\begin{lemma}\label{sgrowthlemma}
If $x$ and $y$ satisfy $x \leq b < x+y \leq c$, and $| V_0(y)| \geq R$, then
\begin{equation*}
\prob \{s(X_1, Y_1) \geq x + y \} \geq  \frac{2b(R- 3x-3y)}{n^2} 
\end{equation*}
\end{lemma}

\section{Growing to $\Theta\left(n^{1/3}\right)$}\label{orderrootnsection}

This section proves Proposition \ref{rootnlemma}. It makes a lot of use of the results of Schramm in ``Compositions of random transpositions'' \cite{SchrammLargeCycles}. A number of definitions are needed to state his main result. 

\begin{defn}
If $(\bar{X}_t)_{t\geq 0}$ is the random transposition walk, define $G_t$ to be the graph on $\{1, 2, \dots, n\}$ such that $\{u, v\}$ is an edge in $G_t$ if and only if the random transposition $(u,v)$ has appeared in the first $t$ steps of our walk. Furthermore, let $W_t$ denote the set of vertices of the largest component of $G_t$. 
\end{defn} 

Note that the behavior of the $W_t$ defined above is well-understood; indeed,by an Erd\H{o}s-R\'{e}nyi theorem (see for example \cite{AlonSpencerBook}), if $t=cn$, then
\begin{equation}\label{Erdos}
\frac{|W_t|}{n} \rightarrow z(2c)
\end{equation}
in probability as $n\rightarrow \infty$, where $z(s)$ is the positive solution of $1 - z = e^{-zs}$.

\begin{defn}\label{PD1distribution}
The Poisson-Dirichlet ($PD(1)$) distribution is a probability measure on the infinite dimensional simplex $\Omega = \{(x_1, x_2, \dots) \left| \right.  \sum_{i=0}^\infty x_i = 1\}$. Sample from this simplex as follows: let $U_1, U_2, \dots$ be an i.i.d sequence of random variables uniform on $[0,1]$. Then, set $x_1 = U_1$, and recursively, 
\begin{equation*}
x_j = U_j\left(1 - \sum_{i=1}^{j-1} x_i\right)
\end{equation*}
Let $(y_i)$ be the $(x_i)$ sorted in nonincreasing order; then, the $PD(1)$ distribution is defined as the law of $(y_i)$. 
\end{defn} 

The main theorem (Theorem 1.1) of Schramm's paper \cite{SchrammLargeCycles} can now be stated. This remarkable result was proved using the tools of graph theory and coupling. A clever lemma showing that vertices that start in `sufficiently large' cycles are likely to end up in cycles of order $n$ also played a pivotal role (Lemma \ref{Schrammlemma} below is an almost exact reproduction of the result.) The full strength of the result is not needed: while Schramm determines the law of the large parts of $X_t$, the only fact necessary here is that after a sufficiently long time, these cycles are of order $n$. For this theorem, treat $X_t$ as an infinite vector by adding infinitely many $0$s at the end of it.

\begin{theorem}[Schramm] \label{schrammpapertheorem}
 Let $c > 1/2$, and take $ t= cn$. As $n \rightarrow \infty$, the law of $\frac{X_t}{|W_t|}$ converges weakly to the $PD(1)$ distribution; that is, for every $\epsilon > 0$, if $n$ is sufficiently large and $t\geq cn$, then there is a coupling of $X_t$ and a $PD(1)$ sample $Y$ such that 
\begin{equation}\label{PoisDir}
P\left\{ \left\|Y -  \frac{X_t}{ \left|W_t\right|} \right\|_\infty < \epsilon\right\} > 1-\epsilon
\end{equation}
where $\left\| \cdot \right\|_\infty$ is the standard $l^\infty$ distance. 
\end{theorem}

The proof that follows uses Theorem \ref{schrammpapertheorem} to show that at time $t = n$, more than half the vertices are in cycles of order $n$ with high probability. This is used to `grow' $m(X_t, Y_t)$ to order $n^{1/3}$, after which the same is done for $s(X_t, Y_t)$. The results for $m(X_t, Y_t)$ are needed before the results for $s(X_t, Y_t)$: since $s(X_t, Y_t) \leq m(X_t, Y_t)$, $m(X_t, Y_t)$ constrains the growth of $s(X_t, Y_t)$ from above. Good control on $m$ is needed before tackling $s$. 
  
\begin{lemma}\label{fraclemma}
Let $k$ be a natural number not dependent on $n$. For sufficiently large $n$, that is, for $n > N =  N(k)$, 
\begin{equation*}
\prob \left\{\left|V_n\left(n/k\right)\right| >n/2\right\} \geq 1 - \frac{6}{k}
\end{equation*}
\end{lemma}
\begin{proof}[\textbf{Proof:}] 
For convenience of notation, let $X = (x_1, x_2, \dots)$ be $X_n$, let $Q = (q_1, q_2, \dots)$ be $\frac{X_n}{|W_n|}$, and let $Y = (y_1, y_2, \dots)$ be a $PD(1)$ sample which is coupled with $Q$ to satisfy Theorem \ref{schrammpapertheorem} above. With current notation, 
\begin{equation}\label{Vcurrent}
\left|V_n\left(n/k\right)\right| = \sum_{x_i\geq \frac{n}{k}}x_i
\end{equation} 
For the rest of the proof, fix $\epsilon = \frac{1}{9k}$. First note that Equation \eqref{Erdos} implies that 
\begin{equation*}
\frac{|W_n|}{n} \rightarrow z(2) \approx 0.797 
\end{equation*}
in probability, which means that $\lim_{n\rightarrow \infty} \prob \left\{|W_n|/n < 3/4\right\} = 0$. Since $Q = X_n/|W_n|$, for sufficiently large $n$,
\begin{equation*}
\prob \left\{x_i \geq \frac{3n}{4} q_i \text{ for all } i \right\} >  1 - \epsilon
\end{equation*}
Furthermore, Theorem \ref{schrammpapertheorem} implies that for sufficiently large $n$, 
\begin{equation*}
\prob\left\{ q_i \geq y_i - \epsilon \text{ for all } i \right\} > 1- \epsilon
\end{equation*}
Combining the above two equations, 
\begin{equation}\label{lowerboundxi}
\prob\left\{ x_i \geq \frac{3n}{4}(y_i - \epsilon) \text{ for all } i \right\} > 1- 2\epsilon
\end{equation}
for sufficiently large $n$. 

Thus, to estimate $|V_n(n/k)|$ it suffices to consider the large parts of the $PD(1)$ sample $Y$. To that end, define the random variable
\begin{equation*}
G_Y(x) = \sum_{y_i \geq x} y_i  
\end{equation*}
It is easy to check that $\expect \left[ G_Y(x) \right] = 1 - x$, and therefore $\expect \left[ 1 - G_Y(x) \right] = x$. Thus, Markov's inequality implies that 
\begin{equation*}
\prob \{G_Y(x) \leq 3/4\}= \prob \{1-G_Y(x) \geq 1/4 \} \leq 4x 
\end{equation*}
Recall that $\epsilon = \frac{1}{9k}$. Then, combining the above with Equation \eqref{lowerboundxi}, 
\begin{equation}\label{setlowerbound}
\prob\left\{ x_i \geq \frac{3n}{4}\left( y_i - \epsilon\right) \text{ for all } i,  G_Y\left(\frac{13}{9k}\right) \geq \frac{3}{4}\right\} \geq  1- \frac{6}{k}
\end{equation}
Finally, assume that $x_i \geq \frac{3n}{4}\left(y_i -\epsilon \right)$ for each $i$, and that $ G_Y\left(\frac{13}{9k}\right) \geq \frac{3}{4}$. Then, Equation \eqref{Vcurrent} implies that 
\begin{align}\label{Vbound}
\left|V_n\left(n/k\right)\right| &\geq \sum_{\frac{3n}{4}(y_i - \epsilon) \geq \frac{n}{k}} \frac{3n}{4}(y_i - \epsilon) = \frac{3n}{4} \left( \sum_{y_i \geq 13/9k} y_i  - \sum_{y_i \geq 13/9k} \frac{1}{9k} \right) \nonumber \\
&\geq \frac{3n}{4}\left( G_Y\left(\frac{13}{9k}\right) - \frac{1}{13}\right)\geq \frac{n}{2}
\end{align}
using the fact that there can be at most $\frac{9k}{13}$ values of $y_i$ that are greater than $\frac{13}{9k}$, since the $y_i$ are positive and sum to $1$. Therefore, using Equation \eqref{setlowerbound}, for sufficiently large $n$
\begin{equation*}
\prob \left\{\left|V_n(n/k)\right| \geq \frac{n}{2}\right\} \geq 1-\frac{6}{k}
\end{equation*}
as required. 
\end{proof}
The above lemma is now applied to find a $t$ of order $n$ such that the probability of having $m(X_t,Y_t) \geq n^{1/3}$ is sufficiently high. Lemmas \ref{mshrinklemma} and \ref{mgrowthlemma} give control of $m(X_t, Y_t)$.

\begin{lemma}\label{mrootnlemma}
If $n$ is sufficiently large and $t\geq 5n$, then 
\begin{equation*}
\prob \left\{m(X_t, Y_t) \geq n^{1/3}, \left|V_t\left(n^{1/3}\right)\right| \geq \frac{n}{2}\right\} \geq \frac{4}{5}
\end{equation*} 
\end{lemma} 
\begin{proof}[\textbf{Proof:}]
From Lemma \ref{fraclemma}, at time $t= n$,
\begin{equation}\label{earlierresult}
\prob\left\{\left|V_t(n/k)\right| \geq \frac{n}{2}\right\} \geq 1-\frac{6}{k}
\end{equation}
Average over the possible values of $X_{t-n}$ to conclude that Equation \eqref{earlierresult} also holds for any time $t\geq n$. Now, for convenience of notation, define  
\begin{equation}\label{Stdefn} 
S_t = \left\{ (X_t,Y_t) \text{ s.t.} \left|V_t\left(n^{1/3}\right)\right| \geq \frac{n}{2} \right\}
\end{equation}
For sufficiently large $n$, $n^{1/3} \leq \frac{n}{k}$ for any fixed value of $k$. Fix $\epsilon > 0$. Then, for $t \geq n$ and sufficiently large $n$, Equation \eqref{earlierresult} implies that $\prob(S_t) \geq 1- \epsilon$. Furthermore, define
\begin{align}\label{Atdefn}
A_t &= \left\{(X_t, Y_t) \left| \right.  m(X_t, Y_t)\geq n^{1/3} \right\} 
\end{align}
To find a lower bound for $\prob(A_t \cap S_t)$ for $t\geq 10n$, note that 
\begin{equation}\label{intersectionlowerbound}
\prob (A_t \cap S_t) \geq \prob(A_t) - \prob(S_t^c) \geq \prob(A_t) - \epsilon 
\end{equation} 
and hence it suffices to bound $\prob(A_t)$. This is done using a recursive argument: at each step $t$, calculate the probability that $m(X_t, Y_t)$ was too small, but $m(X_{t+1}, Y_{t+1})$ is large enough, and vice versa. The probability of $A_t$ is shown to grow sufficiently quickly with $t$. 

Start by bounding the probability that  $m(X_{t+1}, Y_{t+1}) < n^{1/3}$ if $m(X_t, Y_t) \geq n^{1/3}$. By Lemma \ref{mshrinklemma} with $x = n^{1/3}$,
\begin{equation*}
\prob\{(X_{t+1}, Y_{t+1}) \notin A_{t+1} \left| \right. (X_t, Y_t) \in A_t\} \leq \frac{2x^2}{n^2} = \frac{2}{n^{4/3}}
\end{equation*}
and therefore
\begin{equation}\label{mlowerbound}
\prob\{(X_{t+1}, Y_{t+1}) \notin A_{t+1}, (X_t, Y_t) \in A_t\} \leq \frac{2}{n^{4/3}} \prob(A_t)
\end{equation}
Now bound the probability that $m(X_t, Y_t) < n^{1/3}$, while $m(X_{t+1}, Y_{t+1}) \geq n^{1/3}$. In order to bound this in a satisfactory way, enough parts of size $n^{1/3}$are needed; accordingly, work with $(X_t, Y_t)\in A_t^c \cap S_t$. If $m(X_t, Y_t) < n^{1/3}$ and $\left| V_t\left(n^{1/3}\right)\right| \geq \frac{n}{2}$, then using Lemma \ref{mgrowthlemma} with $x=0, y = n^{1/3}$, and $R=\frac{n}{2}$,
\begin{align*}
\prob\{(X_{t+1}, Y_{t+1}) \in A_{t+1} \left| \right. (X_t, Y_t) \in A_t^c \cap S_t\} \geq  \frac{2\left(n/2 - 2n^{1/3}\right)}{n^2} \geq \frac{1 - \epsilon}{n}
\end{align*} 
for sufficiently large $n$. Thus, for $t\geq n$, using the fact that $\prob(S_t)\geq 1 - \epsilon$, 
\begin{equation}\label{mupperbound}
\begin{split}
\prob\{(X_{t+1}, Y_{t+1}) \in A_{t+1}, (X_t, Y_t) \notin A_t\} &\geq  \left( \frac{1 - \epsilon}{n}  \right)\prob(A_t^c\cap S_t)  \\
&\geq \left( \frac{1 - \epsilon}{n}  \right) (1 - \prob(A_t)- \epsilon)
\end{split}
\end{equation}
for sufficiently large $n$. Combining Equations \eqref{mlowerbound} and \eqref{mupperbound}, 
\begin{align*}
\prob(A_{t+1}) - \prob(A_t) &\geq -\frac{2}{n^{4/3}} \prob(A_t) + \left( \frac{1 - \epsilon}{n}  \right) (1 - \prob(A_t)- \epsilon) \nonumber \\
&\geq \frac{1 - \prob(A_t) - 3\epsilon}{n} 
\end{align*}
for sufficiently large $n$ and $t\geq n$. Rearranging the above, 
\begin{equation}\label{mimportant}
\left(1 - 3\epsilon - \prob(A_{t+1})\right) \leq \left(1 - \frac{1}{n} \right) \left(1 - 3\epsilon - \prob(A_t)\right)
\end{equation}
and hence using recursion and the lower bound in Equation \eqref{intersectionlowerbound}, 
\begin{align*}
\left(1 - 3\epsilon - \prob(A_t)\right) &\leq \left( 1 - \frac{1}{n}\right)^{t-n} \leq e^{-(t-n)/n} \\
 \Rightarrow \prob(A_t \cap S_t) &\geq 1 - 4\epsilon -  e^{-(t-n)/n}
\end{align*} 
Thus, for $t\geq 5n$, $\prob(A_t \cap S_t) \geq 1 - 4\epsilon - e^{-4} \approx 1 - 4\epsilon - 0.018$, and picking $\epsilon$ appropriately completes the proof. 
\end{proof}

\begin{prop1}
For sufficiently large $n$, and $t\geq 9n$, 
\begin{equation*}
\prob \left\{s(X_t, Y_t) \geq n^{1/3}, \left| V_t\left(n^{1/3}\right)\right| \geq \frac{n}{2}\right\} \geq \frac{1}{2}
\end{equation*} 
\end{prop1}

\begin{proof}[\textbf{Proof:}] This proof is very similar to the one above. Let $t\geq 5n$, and define
\begin{equation*}
R_t = \left\{m(X_t, Y_t) \geq n^{1/3}, \left|V_t\left(n^{1/3}\right)\right| \geq \frac{n}{2}\right\}
\end{equation*}
From the above lemma, $\prob (R_t) \geq  \frac{4}{5}$.  Now, define 
\begin{align*}
C_t &= \left\{(X_t, Y_t) \left| \right. s(X_t, Y_t)\geq n^{1/3}\right\} 
\end{align*}
It is shown below that $\prob(C_t\cap R_t) \geq \frac{1}{2}$, which will clearly suffice. Note that for $t\geq 5n$,   
\begin{equation}\label{sintersectionlowerbound}
\prob(C_t \cap R_t) \geq \prob(C_t) - \frac{1}{5}
\end{equation}
and hence it suffices to find a lower bound on $\prob(C_t)$. As above, this is done by finding recursive bounds on the probability of $C_{t+1}$ given the probability of $C_t$. By Lemma \ref{sshrinklemma} with $x = n^{1/3}$,  
\begin{equation*}
\prob\{(X_{t+1}, Y_{t+1}) \notin C_{t+1} \left| \right. (X_t, Y_t) \in C_t\} \leq\frac{4x^2}{n^2} =  \frac{4}{n^{4/3}}
\end{equation*}
and therefore
\begin{equation}\label{slowerbound}
\prob\{(X_{t+1}, Y_{t+1}) \notin C_{t+1}, (X_t, Y_t) \in C_t\} \leq \frac{4}{n^{4/3}} \prob(C_t)
\end{equation}

Now, assume that $(X_t, Y_t) \in C_t^c \cap R_t$. Then $m(X_t, Y_t) \geq n^{1/3} > s(X_t, Y_t)$ and  $V_t\left(n^{1/3}\right) \geq \frac{n}{2}$. Therefore, using Lemma \ref{sgrowthlemma} with $x=0, y = n^{1/3}$, and $R=\frac{n}{2}$,  
\begin{align*}
\prob\{(X_{t+1}, Y_{t+1}) \in C_{t+1} \left| \right. (X_t, Y_t) \in C_t^c \cap R_t\} \geq  \frac{2\left(n/2 - 3n^{1/3}\right)}{n^2} = \frac{1}{n} - \frac{6}{n^{5/3}}
\end{align*} 
Thus, for $t\geq 5n$, using the fact that $\prob(R_t)\geq \frac{4}{5}$, 
\begin{equation}\label{supperbound}
\begin{split}
\prob\{(X_{t+1}, Y_{t+1}) \in C_{t+1}, (X_t, Y_t) \notin C_t\} &\geq  \left( \frac{1}{n} - \frac{6}{n^{5/3}}  \right)\prob(C_t^c\cap R_t) \\
&\geq \left( \frac{1}{n} - \frac{6}{n^{5/3}}  \right) \left(\frac{4}{5}- \prob(C_t)\right)
\end{split}
\end{equation}
for sufficiently large $n$. Therefore, combining Equations \eqref{slowerbound} and \eqref{supperbound} and picking $n$ sufficiently large, 
\begin{equation}\label{simportant}
\begin{split}
\prob(C_{t+1}) - \prob(C_t) &\geq -\frac{2}{n^{4/3}} \prob(C_t) + \left( \frac{1}{n} - \frac{6}{n^{5/3}}  \right) \left(\frac{4}{5}- \prob(C_t) \right)  \\
&\geq \frac{3/4- \prob(C_t)}{n}
\end{split}
\end{equation}
for $t\geq 5n$. Rearranging analogously to Equation \eqref{mimportant}, 
\begin{equation*}
\left( \frac{3}{4} - \prob(C_{t+1})\right)\leq \left(1- \frac{1}{n} \right)\left( \frac{3}{4} - \prob(C_t)\right)
\end{equation*}
As before, for $t\geq 9n$, $\prob(C_t) \geq \frac{7}{10}$. Combining this with Equation \eqref{sintersectionlowerbound},  
\begin{equation*}
\prob(C_t \cap R_t) \geq \frac{1}{2}
\end{equation*}
for $t\geq 9n$ and $n$ sufficiently large, as required. 
\end{proof}

\section{Growing to $\Theta(n)$}\label{ordernsection}

This section proves Proposition \ref{sbiglemma}, which shows that $s(X_t, Y_t)$ can be grown to order $n$. This section is structured similarly to the previous one: proving a lemma about overall cycle sizes, then a lemma about $m(X_t, Y_t)$, and then finally Proposition \ref{sbiglemma}. Again, use is made of the technical results in Lemmas \ref{schrammshrinklemma} through \ref{sgrowthlemma}. 

The idea behind the proof is largely based on Lemma 2.3 from ``Compositions of random transpositions'' \cite{SchrammLargeCycles}. Let $\epsilon, \delta$ and $j$ be chosen as in Proposition \ref{sbiglemma}. Recall that Definition \ref{Kandtaudef} defines $K =  \lceil \log_2(\epsilon \delta n) \rceil 
$ and
\begin{equation*}
a_r = \lceil 2 \delta^{-1} 2^{-r} n (\log_2 n - r)\rceil \ \text{ and }\ 
\tau_r = \sum_{i=j}^{r-1} a_i
\end{equation*}
for $r$ between $j$ and $K$, with $\tau_j = 0$. Then, define
\begin{equation}\label{intervaldefn}
I_r = [\tau_r, \tau_{r+1}-1]
\end{equation}
and for convenience of notation, define $I_K = \{\tau_K\}$. 

As should be clear from the statement of Proposition \ref{sbiglemma}, the argument starts with $s(\pi, \sigma) \geq 2^{j+1}$ and $V_0(2^{j+1}) \geq \delta n$, and shows that at time $\tau_K$, the probability that $s(X_{\tau_K}, Y_{\tau_K})$ is less than $\epsilon \delta n$ is appropriately bounded above. In fact, something stronger is shown: for the intervals $I_r$ as defined above, one `expects' to have 
\begin{equation*}
V_t (2^{r+1}) \geq \frac{\delta n}{2}, s(X_t^, Y_t) \geq 2^r, \text{ and } m(X_t, Y_t) \geq 2^{r+1}
\end{equation*}
for all $r$ between $j$ and $K$. This would clearly suffice to prove the result. 

The first lemma is almost identical to Lemma 2.3 from \cite{SchrammLargeCycles} -- it is reproven here for completeness, and to illustrate the technique. This lemma starts with $\sigma\in S_n$, and $\left|V_0(2^{j+1})\right| \geq \delta n$. It gives an upper bound for the expected number of vertices that start in cycles of size at least $2^{j+1}$, and that are not in cycles of size $\epsilon \delta n$ at time $\tau_K$. This shows that `most' vertices that start in cycles of size $2^{j+1}$ are in cycles of order $n$ at time $\tau_K$. 

\begin{lemma}\label{Schrammlemma}
Let $\sigma \in S_n$. Let $\delta \in (0,1)$ be a constant such that $\left|V_0(2^{j+1})\right| \geq \delta n$, and let $K$ and $\tau_K$ be defined as they are above and in Definition \ref{Kandtaudef}. Fix $\epsilon \in (0, 1/32)$. For the random transposition walk $\left( \bar{X}_t\right)_{t\geq 0}$,
\begin{equation}\label{schrammlemma} 
\expect \left|V_0(2^{j+1})\setminus V_{\tau_K}(2\epsilon \delta n)\right| \leq O(1) \delta^{-1} \epsilon \left| \log(\epsilon \delta)\right| n 
\end{equation} 
where the constant implied in the $O(1)$ notation is universal.
\end{lemma}
\begin{proof}[\textbf{Proof:}]
Before beginning the proof, consider what is being shown. Starting with a $\sigma$ such that $\left|V_0(2^{j+1})\right| > \delta n$ means that at least $\delta n$ of the vertices in $\sigma$ are in cycles of size at least $2^{j+1}$. An upper bound on the expected size of $V_0(2^{j+1})\setminus V_{\tau_K}(2 \epsilon \delta n)$ is needed: that is, an upper bound on the expected number of vertices that started off in cycles of size at least $2^{j+1}$ in $\sigma$, and ended up in cycles of size less than $2\epsilon \delta n$ at time $\tau_K$. 

Something stronger is shown: conditioned on $v\in V_0(2^{j+1})$,
\begin{equation}\label{strongbound}
\begin{split}
&\expect \left|\left\{ v \text{ s.t. } C_t(v) < 2^{r+1} \text{ for any } t \in I_r, \text{ for }  r \in [j, K] \right\}\right| \leq\\
 & \hspace{200 pt} O(1) \delta^{-1} \epsilon \left| \log(\epsilon \delta)\right| n
\end{split}
\end{equation}
This requires an upper bound on the expected number of vertices that for any time $t \in I_r$ are `too small' for $I_r$: they are of size less than $2^{r+1}$. Note that the above set includes all vertices such that $C_{\tau_K}(v) < 2\epsilon \delta n \leq 2^{K+1}$, and hence the above bound would suffice. 

Three different possibilities are considered. First of all, an upper bound is needed on the expected number of vertices $v$ such that at any point, the cycle containing $v$ is split, and becomes too small. Secondly, all vertices that appear in permutations with an insufficient number of large parts are rejected. And thirdly, it is necessary to bound the possibility that the cycle containing $v$ does not grow sufficiently during $I_r$. Call the vertices that fall into any of these undesirable categories `failed.'

In the next three sections, condition on $v \in V_0(2^{j+1})$: that is, assume that $v$ is in a cycle of size $2^{j+1}$ in $\sigma$. This means that $v$ has not failed at time $0$. 

\paragraph{The cycle containing $v$ becomes too small} Let $r \in [j, K-1]$, and let $t \in I_r + 1 = [\tau_r+1, \tau_{r+1}]$. For $C_t(v)$ to be of size $2^{r+2}$ by time $\tau_{r+1}$, calculate the probability that for any $t \in I_r +1$, the cycle containing $v$ is split, and $v$ is then contained in a cycle of size less than $2^{r+2}$. To be precise, define $F_t$ to be the set of vertices at time $t$ such that $\left|C_t(v)\right| < \left|C_{t-1}(v)\right|$ and $\left|C_t(v)\right| < 2^{r+2}$. Find the expected size of $F_t$: by definition, this is the expected number of vertices $v$, whose cycle is split from time $t-1$ to time $t$, and which are in cycles of size less than $2^{r+2}$ at time $t$. By Lemma \ref{schrammshrinklemma}, 
\begin{equation*}
\expect \left|F_t\right| \leq \frac{2\left(2^{r+2}\right)^2}{n} = \frac{2^{2r+5}}{n} 
\end{equation*}

Now, define the cumulative set $\tilde{F}_t = \bigcup_{x = 1}^t F_x$. This is the set of all vertices up to time $t$, whose cycles have at any time $x \leq t$ been split into ones that are `too small.' Clearly,  
\begin{equation} \label{firstFineq}
\begin{split}
\expect \left|\tilde{F}_{\tau_K}\right| &\leq \sum_{x = 1}^{\tau_K} \expect \left| F_x\right| \leq \sum_{r=j}^{K-1} a_r \frac{2^{2r+5}}{n}  \\
&\leq \sum_{r = j}^{K-1}\lceil 2 \delta^{-1} 2^{-r} n (\log_2 n - r)\rceil \frac{2^{2r+5}}{n}  \\
&\leq \sum_{r = j}^{K-1}\left( 2 \delta^{-1} 2^{-r} n (\log_2 n - r) + 1\right)\frac{2^{2r+5}}{n} \\ 
&\leq \sum_{r=j}^{K-1} 2^6 \delta^{-1} 2^r (\log_2 n  - r) + \sum_{r=j}^{K-1}\frac{2^{2r+5}}{n}
\end{split}
\end{equation}
Now, 
\begin{align*}
\sum_{r=j}^{K-1} r 2^r &= (K-2)2^K - (j-2)2^j \\
\sum_{r=j}^{K-1} 2^r &= 2^K - 2^j
\end{align*}
shows that  
\begin{equation} \label{Fbound}
\expect \left|\tilde{F}_{\tau_K}\right| \leq 2^8\left| \log_2 (\epsilon \delta)\right| \epsilon  n  
\end{equation}
\paragraph{Permutations with insufficiently many large parts} It is also necessary to rule out vertices in permutations for which the union of the `large parts' isn't sufficiently high. This will be useful for the next part of the proof. To be more precise, let $t \in I_r$: if $\left|V_t(2^{r+1})\right| < \delta n/2$, and this is the first $t$ for which the inequality holds, then consider all vertices in $X_t$ to have failed, and set $H_t = \{1,\dots, n\}$. Otherwise, set $H_t = \emptyset$. 

Again, define the cumulative set $\tilde{H}_t = \bigcup_{x = 0}^t H_x$. This is the union of all vertices that up to time $t$ have been in a permutation with insufficiently many large parts, by the above definition. It is clear that this set is either empty, or contains all the vertices. There is no current available upper bound on the expectation for $\tilde{H}_t$; one will be derived after the next section of the proof. 

\paragraph{The cycle containing $v$ doesn't grow sufficiently} Next, consider how a vertex $v$ might fail at time $t$, if it does not fall into $\tilde{F}_t$ or $\tilde{H}_{t-1}$. Assume $t$ is the minimal time for which $v$ fails: since failed vertices include all vertices contained in cycles that are `too small', if $s < t$ and $s \in I_{k}$ then $\left| C_s(v)\right| \geq 2^{k +1}$. Now, assume that $t$, the first time at which $v$ fails, is in $I_r$: thus, $t-1$ is either in $I_{r-1}$ or in $I_r$. Either way, since it was assumed that $v$ is not in $F_t^\pi$, it can't be that $\left| C_t(v)\right| < \left| C_{t-1}(v)\right|$ and $\left|C_t(v)\right|< 2^{r+1}$. Since the vertex $v$ fails at time $t$, $C_t(v)$ must contain fewer than $2^{r+1}$ vertices. Combine this with the preceding statement to conclude that $C_{t-1}(v)$ also contains fewer than $2^{r+1}$ vertices. However, by definition the vertex $v$ did not fail at time $t-1$. This implies $t-1$ must have been in $I_{r-1}$. Thus, the only remaining times at which vertices could fail are $t = \tau_r$, for $r \in \{j, j+1, \dots, K\}$. Having conditioned on $v \in V_0(2^{j+1})$, it may be concluded that $v$ can't fail at time $\tau_j = 0$.

Now, define $B_r$ to be the set of vertices at time $\tau_r$ that are not in $\tilde{F}^{\tau_r} \cup \tilde{H}^{\tau_r-1}$, such that $\left|C_{\tau_r}(v)\right| < 2^{r+1}$, and that have not failed previously. As before, define $\tilde{B}_r = \bigcup_{x = j}^r B_x$ and estimate the expected size of $B_r$. 

Condition on $v \notin \tilde{F}_{\tau_r} \cup \tilde{H}_{\tau_r-1}$ and calculate the probability that $v$ fails at $\tau_r$, given that it has not failed up to that time. First, for $t\in I_{r-1}$, $\left|C_t(v) \right|\geq 2^r$. Furthermore, since $v \notin \tilde{F}_{\tau_r}$, there was no time between $\tau_{r-1}$ and $\tau_r$ at which the cycle containing $v$ was split to contain fewer than $2^{r+1}$ vertices. This means that if $v$ failed at time $\tau_r$, then $C_t(v)$ must have been of size less than $2^{r+1}$ for all $t \in [\tau_{r-1}, \tau_r - 1]$. Therefore, for $t\in I_{r-1}$,
\begin{equation}\label{cyclebound}
 2^r \leq \left| C_t(v) \right| < 2^{r+1}
\end{equation}

Furthermore, since $v$ is not in $\tilde{H}_t$ for any $t \in I_{r-1}$, for every $t \in I_{r-1}$, $\left|V_t (2^r)\right| \geq \delta n/2$. Consider the probability that from time $t$ to time $t+1$, the cycle containing $v$ is merged with a cycle of size at least $2^r$. By \eqref{cyclebound} above, the size of $C_t(v)$ is at least $2^r$, so such a merge would result in $C_{t+1}(v) \geq 2^{r+1}$. Using the above reasoning implies that $\left|C_{\tau_r}(v)\right| \geq 2^{r+1}$, and therefore $v$ does not fail at time $\tau_r$. Now, again by \eqref{cyclebound}, the cycle containing $v$ is of size at most $2^{r+1}$. Since $\left|V_t(2^r)\right| \geq \delta n/2$, this means the union of the cycles disjoint from $C_t(v)$ of size at least $2^r$ contains at least $\delta n/2 - 2^{r+1}$ vertices. Now, since $r \leq K = \lceil\log_2(\epsilon \delta n)\rceil$, $2^{r+1} \leq 2^{K+1} \leq 4\epsilon\delta n$, and since $\epsilon < 1/32$,
\begin{equation*}
\frac{\delta n}{2} - 2^{r+1} \geq \frac{\delta n}{4}
\end{equation*} 
Thus, the union of the cycles of size at least $2^r$ disjoint from $C_t(v)$ is of size at least $\delta n/4$, and therefore 
\begin{equation*}
\prob \{\text{$C_t(v)$ merges with a cycle of size $\geq 2^r$}\} \geq 2\cdot 2^r\frac{\delta n/4}{n^2} = 2^{r-1}\delta n^{-1}
\end{equation*}
Clearly, for $v$ to be in $B_r$, it cannot be that $C_t(v)$ merges with a cycle of size $\geq 2^r$ for any $t \in I_{r-1}$. Therefore,  
\begin{align} \label{firstbbound}
\prob \{v \in B_r\} &\leq \left(1 - 2^{r-1}\delta n^{-1}\right)^{a_{r-1}}\\
               &\leq \exp(-2^{r-1}\delta n^{-1} a_{r-1}) \nonumber
\end{align}
and since $a_{r-1} \geq 2 \delta^{-1} 2^{-r+1} n (\log_2 n - r+1)$, 
\begin{equation*}
\prob \{v \in B_r\} \leq e^{2(r - 1 - \log_2 n)}
\end{equation*}
Now, $r - 1 \leq K - 1 \leq \log_2(\epsilon \delta n)$, and therefore, $\log_2 n - r + 1 \leq 0$. Thus, 
\begin{equation*}
\prob \{v \in B_r\} \leq e^{2(\log_2 n - r +1)} \leq 2^{\log_2 n -r + 1} =  \frac{2^{r-1}}{n}
\end{equation*}
This yields
\begin{equation*}
\expect \left| B_r \right| \leq n \prob \{v \in B_r\} = 2^{r-1}
\end{equation*}
and therefore, 
\begin{align}\label{Bbound}
\expect \left| \tilde{B}_K \right| &\leq \sum_{r=j}^{K} \expect \left| B_r\right| \leq \sum_{r=j}^{K} 2^{r-1}\nonumber \\
 &\leq 2^K \leq \epsilon \delta n + 1
\end{align}

Finally, bound the expected size of $H_t$, the set of vertices in permutations with insufficiently many large parts. Recall that for $t \in I_r$, if $\left|V_t(2^{r+1})\right| < \delta n/2$ and $t$ was the first time this inequality held, $H_t$ was defined to be the set of all vertices, and it was otherwise defined to be the empty set. If $H_t$ is non-empty, then the set of vertices in $X_t$ that are in cycles of size less than $2^{r+1}$ has size at least $n- \delta n/2 \geq \delta n/2$. Now, consider $v$ in $H_t$ such that $\left|C_t(v)\right| < 2^{r+1}$. By definition, $v$ has failed by time $t$, and  $v$ is not in $H_s$ for any $s < t$. Therefore, each such vertex is in $\tilde{F}_t \cup \tilde{B}_r$. Thus, 
\begin{equation*}
\expect \left|\tilde{H}_{\tau_K}\right| \leq n \prob \left\{\left|\tilde{F}_{\tau_K} \cup \tilde{B}_K\right| \geq \delta n/2 \right\} \leq 2 \delta^{-1} \expect \left|\tilde{F}_{\tau_K} \cup \tilde{B}_K\right|
\end{equation*}
so using \eqref{Fbound} and \eqref{Bbound} above, 
\begin{equation}\label{Hbound}
\begin{split}
\expect \left|\tilde{H}_{\tau_K}\right| &\leq 2\delta^{-1}(\epsilon \delta n + 1 + 2^8\left| \log (\epsilon \delta)\right| \epsilon  n)  \\
&\leq (2^9 + 1) \epsilon \delta^{-1} \left|\log_2(\epsilon \delta)\right| n 
\end{split}  
\end{equation}
as desired. Finally, adding up the expectations for $H_{\tau_K}, B_K$ and $F_{\tau_K}$ in \eqref{Hbound}, \eqref{Bbound} and \eqref{Fbound} completes the proof.  
\end{proof}

The next lemma is similar. It shows that $m(X_t, Y_t)$ becomes sufficiently large at time $\tau_K$. The proof is almost entirely analogous; the only substantial difference is in the bound for the probability of $X_t$ having insufficiently many `large parts.' For this bound, Equation \eqref{Hbound} above has to be used. Lemmas \ref{mshrinklemma} and \ref{mgrowthlemma} will also be used. 

\begin{lemma}\label{mbiglemma}
Assume $\rho(\sigma, \tau) = 1$. Let $j$ be a natural number such that $m(\sigma, \tau) \geq 2^{j+1}$, and let $\delta \in (0,1]$ be a constant such that $\left|V_0(2^{j+1})\right| \geq \delta n$ and $m(\pi, \sigma) \geq 2^{j+1}$. Let $K$ and $\tau_K$ be defined as above, and let $\epsilon \in (0, 1/16)$. Then, 
\begin{equation}\label{provemgrow} 
\prob \{m(X_{\tau_K}, Y_{\tau_K}) < 2\epsilon \delta n\} \leq O(1) \delta^{-1} \epsilon \left| \log(\epsilon \delta)\right|
\end{equation} 
where the constant implied in the $O(1)$ notation is universal.
\end{lemma}
\begin{proof}[\textbf{Proof:}]
This proof is almost exactly analogous to the previous one, except that instead of keeping track of failed vertices, failed pairs of partitions will be tracked. Something stronger is shown: 
\begin{equation}\label{strongmbound}
\prob \{m(X_t, Y_t) < 2^{r+1}\text{ for any } t \in I_r, \text{ for } j \in [j, K]\} \leq O(1) \delta^{-1} \epsilon \left| \log(\epsilon \delta)\right|
\end{equation}
Again, the argument requires upper bounds on three different cases: the one where $m(X_t, Y_t)$ shrinks to become too small at time $t$, the one where $X_t$ doesn't have sufficiently many large parts, and the one where $m(X_t, Y_t)$ fails to grow sufficiently during $I_r$. The only major difference in the proof is use of the bound from Lemma \ref{Schrammlemma} to bound the probability of $X_t$ having insufficiently many large parts.

Since the quantities specified are precisely analogous, use the names $\mathcal{F}_t, \mathcal{B}_t$ and $\mathcal{H}_t$. 

\paragraph{Probability $m(X_t, Y_t)$ gets too small during $I_r$} For $t\in I_r + 1$, define $\mathcal{F}_t$ to be the set of pairs $(X_t, Y_t)$ such that $m(X_t, Y_t) < m(X_{t-1}, X_{t-1})$ and $m(X_t, Y_t) < 2^{r+2}$. Apply Lemma \ref{mshrinklemma} above. Let $x = \min(2^{r+2}, m(X_{t-1}, Y_{t-1}))$. Then, $x \leq m(X_{t-1}, Y_{t-1})$, and therefore from Lemma \ref{mshrinklemma}, the probability that $m(X_t, Y_t)$ is less than $x$ is bounded above by $\frac{2x^2}{n^2}$. By definition of $\mathcal{F}_t$, this means that 
\begin{equation*}
\prob \{\mathcal{F}_t\} \leq \frac{2x^2}{n^2} \leq 2\frac{\left(2^{r+2}\right)^2}{n^2} = \frac{2^{2r +5}}{n^2}
\end{equation*}
Define the cumulative set $\tilde{\mathcal{F}}_t = \bigcup_{x = 1}^t \mathcal{F}_x$. Therefore, 
\begin{equation*}
\prob \{\tilde{\mathcal{F}}_{\tau_K}\} \leq \sum_{x=1}^{\tau_K} \prob \{\mathcal{F}_x\} \leq  \sum_{r=j}^{K-1} a_r\frac{2^{2r+5}}{n^2}
\end{equation*}
and doing a calculation almost identical to \eqref{Fbound}, 
\begin{equation}\label{mFbound}
\prob \{\tilde{\mathcal{F}}_{\tau_K}\} \leq  2^8 \epsilon \left| \log_2 (\epsilon \delta)\right|
\end{equation}
Note that the only difference in the calculation was an extra factor of $n$ in the denominator.
\paragraph{Probability $X_t$ doesn't have enough large parts} Define $\mathcal{H}_t$ almost exactly as $H_t$ in the last lemma, except that instead of making it a set of vertices, let it be a set of pairs $(X_t, Y_t)$. $(X_t, Y_t)$ is included in $\mathcal{H}_t$ precisely when $X_t$ doesn't have enough large parts: that is, if $t \in I_r$, then $(X_t, Y_t)$ is in $\mathcal{H}_t$ if $\left|V_t(2^{r+1})\right| < \delta n/2$, and $t$ is the first time for which this inequality holds. Define $\tilde{\mathcal{H}_t}$ as usual to be the cumulative set. 

Clearly, if $(X_t, Y_t) \in \mathcal{H}_t$, then $H_t$ contains $n$ vertices, and otherwise $H_t$ is empty. Since $\left|V_0(2^{j+1})\right| \geq \delta n$, the results derived in Lemma \ref{Schrammlemma} can be used. Therefore,  
\begin{equation*}
\prob \{\mathcal{H}_t\} = \frac{1}{n}\expect \left| H_t\right|
\end{equation*}
and thus, from \eqref{Hbound} above, 
\begin{equation}\label{mHbound}
\prob \{\mathcal{H}_t\} \leq (2^9 + 1) \epsilon \delta^{-1} \left|\log_2(\epsilon \delta)\right| 
\end{equation}

\paragraph{Probability $m(X_t, Y_t)$ doesn't grow sufficiently during $I_r$} As before, the only remaining times that $(X_t, Y_t)$ can fail is at times $\tau_r$. Accordingly, define $\mathcal{B}_r$ to be those pairs $(X_{\tau_r}, Y_{\tau_r})$ that are not in $\tilde{\mathcal{F}}_{\tau_r}$ or $\tilde{\mathcal{H}}_{\tau_r - 1}$, such that $m(X_{\tau_r}, Y_{\tau_r}) < 2^{r+1}$ and that have not failed previously. As before, if $(X_{\tau_r}, Y_{\tau_r})$ is in $\mathcal{B}_r$, then it had not failed in $I_{r-1}$, and therefore, for $t\in I_{r-1}$, $m(X_t, Y_t) \geq 2^r$. Furthermore, since $(X_{\tau_r}, Y_{\tau_r})$ is not in $\tilde{\mathcal{F}}_{\tau_r}$, it must be that $m(X_t, Y_t)$ is less than $2^{r+1}$ for $t\in I_{r-1}$. Thus, for $t\in I_{r-1}$, 
\begin{equation}\label{mbound}
2^r \leq m(X_t, Y_t) < 2^{r+1}
\end{equation}
Furthermore, since $\mathcal{B}_r$ is disjoint from $\mathcal{H}_{\tau_r -1}$, for every $t\in I_{r-1}$, $\left|V_t(2^r)\right| \geq \delta n/2$. Since $m(X_t, Y_t) \geq 2^r$, Lemma \ref{mgrowthlemma} holds with $R = \delta n/2$ and $x = y = 2^r$. Let $c = m(X_t, Y_t)$. Thus, for  any $t \in I_r$,  
\begin{equation*}
\prob \{m(X_{t+1}, Y_{t+1}) \geq 2^{r+1}\} \geq \frac{2c(\delta n/2 - 2c)}{n^2} \geq \frac{2^{r+1}(\delta n/2 - 2^{r+2})}{n^2}
\end{equation*}
and since $\epsilon < \frac{1}{32}$, and $r \leq K \leq \log_2(\epsilon \delta n) + 1$, $\delta n/2 - 2^{r+2} \geq \delta n/4$. Thus, 
\begin{equation*}
\prob \{m(X_{t+1}, Y_{t+1}) \geq 2^{r+1}\} \geq 2^{r-1} \delta n^{-1}
\end{equation*}
Finally, the probability of $\mathcal{B}_r$ is the probability that $m(X_{t+1}, Y_{t+1})$ isn't at least $2^{r+1}$ for any $t \in I_r$, and therefore,
\begin{equation*} 
\prob \{\mathcal{B}_r\}\leq \left(1 - 2^{r-1}\delta n^{-1}\right)^{a_{r-1}}
\end{equation*}
and since this is precisely the same inequality as in \eqref{firstbbound},  
\begin{equation*}
\prob \{\mathcal{B}_r\} \leq \frac{2^{r-1}}{n}
\end{equation*}
and hence
\begin{equation}\label{mBbound}
\prob \{\tilde{\mathcal{B}}_{K}\} \leq \sum_{r=j}^K \prob \{\mathcal{B}_r\} \leq \sum_{r=j}^K \frac{2^{r-1}}{n} \leq \epsilon \delta + \frac{1}{n}
\end{equation}
Thus, adding \eqref{mFbound}, \eqref{mHbound}, and \eqref{mBbound}, 
\begin{equation} \label{mfinalbound}
\begin{split}
&\prob \{m(X_t, Y_t) < 2^{r+1}\text{ for any } t \in I_r\text{, for } r\in [j, K]\} 
\leq \\
&\hspace{170 pt} (2^9 + 2^8+1) \epsilon \delta^{-1} \left|\log_2(\epsilon \delta)\right| 
\end{split}
\end{equation}
which is what is needed. 
\end{proof}

The stage is almost set to prove an analogous result for $s(X_t, Y_t)$. As above, the two technical Lemma \ref{sshrinklemma} and \ref{sgrowthlemma} are used. As in the previous section, $m(X_t, Y_t)$ must be `sufficiently large' to allow $s(X_t, Y_t)$ to grow. This is the reason for proving the lemma concerning $m(X_t, Y_t)$ first. 

\begin{prop2}
Let $(X_t, Y_t)$ be the usual coupling started at $(X_0, Y_0) = (\sigma, \tau)$, where $\rho(\sigma, \tau) \leq 1$. Let $j$ be a number and let $\delta \in (0,1]$ be a constant such that $\left|V_0(2^{j+1})\right| \geq \delta n$ and $s(\sigma, \tau) \geq 2^{j+1}$. If $K$ and $\tau_K$ are defined as in Definition \ref{Kandtaudef} and $\epsilon \in (0, 1/32)$, then
\begin{equation*}
\prob\{s(X_{\tau_K}, Y_{\tau_K}) < \epsilon \delta n\} \leq O(1) \delta^{-1} \epsilon \left| \log(\epsilon \delta)\right| 
\end{equation*} 
where the constant implied in the $O(1)$ notation is universal.
\end{prop2}

\begin{proof}[\textbf{Proof of Lemma \ref{sbiglemma}}]
This proof is analogous to the proof of Lemma \ref{Schrammlemma} and \ref{mbiglemma}, except that the previous two lemmas are used to bound the probability that $s(X_t, Y_t)$ shrinks or grows. As before, a stronger statement is proved: 
\begin{equation}\label{strongsbound}
\prob \{s(X_t, Y_t) < 2^ r\text{ for any } t \in I_r \text{, for } r\in [j, K]\} \leq O(1) \delta^{-1} \epsilon \left| \log(\epsilon \delta)\right|
\end{equation}

Again, bounds are needed for a number of different cases: for the probability that $s(X_t, Y_t)$ shrinks to become too small during $I_r$, the probability that $X_t$ doesn't have enough large parts, and that the probability that $s(X_t, Y_t)$ doesn't grow sufficiently on $I_r$. Furthermore, note that Lemma \ref{sgrowthlemma} requires the assumption that $m(\sigma, \tau) \geq 2x$ to lower bound on the probability that $s(X_1, Y_1) \geq 2x$. Since $s(X_t, Y_t)$ must grow during $I_r$ to be at least $2^{r+1}$ by $\tau_{r+1}$, $m(X_t, Y_t)$ must be at least $2^{r+1}$ on $I_r$. Lemma \ref{mbiglemma} is used to bound the probability that $m(X_t, Y_t)$ is too small. 

The quantities are precisely analogous to the ones in the two similar previous lemmas. Accordingly, name them $\mathscr{F}_t, \mathscr{H}_t$, and $\mathscr{B}_t$, using the same letters but yet another font. The new quantity $\mathscr{M}_t$ is added, as discussed above. 

\paragraph{Probability $s(X_t, Y_t)$ gets too small during $I_r$} For $t\in I_r +1 = [\tau_r +1, \tau_{r+1}]$, define $\mathscr{F}_t$ to be the set of pairs $(X_t, Y_t)$ such that $s(X_t,Y_t) < s(X_{t-1}, Y_{t-1})$ and $s(X_t, Y_t) < 2^{r+1}$. Apply Lemma \ref{sshrinklemma} above. Define $x = \min(2^{r+1}, s(X_{t-1}, Y_{t-1}))$. Then, $x\leq s(X_{t-1}, Y_{t-1})$, and therefore Lemma \ref{sshrinklemma} applies. Plugging it in, the probability that $s(X_t, Y_t)$ is less than $x$ is at most $\frac{4x^2}{n^2}$. Thus, 
\begin{equation*}
\prob \{\mathscr{F}_t\} \leq \frac{4x^2}{n^2} \leq \frac{4(2^{r+1})^2}{n^2} = \frac{2^{2r + 4}}{n^2}
\end{equation*}
Now, define the cumulative set $\tilde{\mathscr{F}}_t = \bigcup_{x=1}^t \mathscr{F}_x$. Then, 
\begin{equation*}
\prob \{\tilde{\mathscr{F}}_{\tau_K}\} \leq \sum_{x = 1}^{\tau_K} \prob \{\mathscr{F}_x\} \leq 
\sum_{r=j}^{K-1} a_r \frac{2^{2r +4}}{n^2}
\end{equation*}
Doing a calculation identical to the one in \eqref{Fbound} and \eqref{mFbound}, 
\begin{equation}\label{sFbound}
\prob \{\tilde{\mathscr{F}}_{\tau_K}\}\leq 2^7\epsilon \left| \log_2(\epsilon \delta)\right|
\end{equation}
\paragraph{Probability $X_t$ doesn't have enough large parts} For $t\in I_r$, define $\mathscr{H}_t$ very similarly to before, to be the set of $(X_t, Y_t)$ such that $\left|V_t(2^r)\right| < \delta n/2$, whenever this is the first $t$ for which this inequality holds. Define $\tilde{\mathscr{H}}_t$ to be the usual cumulative set. Now, from Lemma \ref{mbiglemma}, 
\begin{equation*}
\mathcal{H}_t = \{(X_t, Y_t) \left| \right. \left|V_t(2^{r+1})\right| < \delta n/2\}
\end{equation*}
Since $V_t(2^r) \supseteq V_t(2^{r+1})$, clearly $\mathscr{H}_t \supseteq \mathcal{H}_t$, and therefore, using \eqref{mHbound}
\begin{equation}\label{sHbound}
\prob \{\tilde{\mathscr{H}}_t\} \leq \prob \{\tilde{\mathcal{H}}_t\} \leq (2^9 +1)\epsilon \delta^{-1}\left|\log_2(\epsilon \delta)\right|
\end{equation}
\paragraph{Probability $m(X_t, Y_t)$ is too small} For $t\in I_r$, define $\mathscr{M}_t$ to be the set of all $(X_t, Y_t)$ such that $m(X_t, Y_t) < 2^{r+1}$. As usual, define $\tilde{\mathscr{M}}_t$ to be the cumulative set. Since at the start $s(\sigma, \tau)\geq 2^{j+1}$,  $m(\sigma, \tau)\geq  2^{j+1}$ is forced. By assumption, $V_0(2^{j+1}) \geq \delta n$, so any inequalities derived in Lemma \ref{mbiglemma} are in force. Thus, from Equation \eqref{mfinalbound},
\begin{equation*}
\prob \{m(X_t, Y_t) < 2^{r+1}\text{ for any } t \in I_r\text{, for } r\in [j, K]\} \leq (2^9 + 2^8+1) \epsilon \delta^{-1} \left|\log_2(\epsilon \delta)\right| 
\end{equation*}
and clearly, from the definition of $\tilde{\mathscr{M}}^t$,  
\begin{equation}\label{sMbound}
\prob \{\tilde{\mathscr{M}}^t\} \leq  (2^9 + 2^8+1) \epsilon \delta^{-1} \left|\log_2(\epsilon \delta)\right| 
\end{equation}

\paragraph{Probability $s(X_t, Y_t)$ doesn't grow sufficiently during $I_r$} As before, the only remaining times that $s(X_t,Y_t)$ can fail is at time $\tau_r$. Therefore, define $\mathscr{B}_r$ to be the set of $(X_{\tau_r},Y_{\tau_r})$ that are not in $\tilde{\mathscr{F}}_{\tau_r}$, $\tilde{\mathscr{H}}_{\tau_r - 1}$ or $\tilde{\mathscr{M}}_{\tau_r}$, such that $s(X_{\tau_r}, Y_{\tau_r}) < 2^r$ and that have not failed previously. If $(X_{\tau_r}, Y_{\tau_r})$ is in $\mathscr{B}_r$, then it had not failed in $I_{r-1}$, and therefore for $t\in I_{r-1}$, $s(X_t, Y_t) \geq 2^{r-1}$. Furthermore, since $(X_{\tau_r}, Y_{\tau_r})$ is not in $\tilde{\mathscr{F}}_{\tau_r}$, for $t\in I_{r-1}$, $s(X_t, Y_t) < 2^r$. Thus, for $t\in I_{r-1}$,
\begin{equation}\label{sbound}
2^{r-1} \leq s(X_t, Y_t) < 2^r
\end{equation}
Furthermore, since $(X_{\tau_r}^\pi, X_{\tau_r}^\sigma)$ is not in $\tilde{\mathscr{M}}_{\tau_r}$, for $t \in I_{r-1}$ 
\begin{equation*}
m(X_t, Y_t) \geq 2^r
\end{equation*}
Finally, since $\mathscr{B}_r$ is disjoint from $\mathscr{H}_{\tau_r-1}$, for every $t\in I_{r-1}$, $\left| V_t(2^{r-1})\right| \geq \delta n/2$. Now apply Lemma \ref{sgrowthlemma} with $R = \delta n/2$ and $x = y = 2^{r-1}$. For any $t\in I_r$, 
\begin{equation*}
\prob \{s(X_{t+1}, Y_{t+1}) \geq 2^r\} \geq \frac{2x(R-3x - 3y)}{n^2} = \frac{2^r(\delta n/2 -3\cdot 2^r)}{n^2}
\end{equation*}
Since $r \leq K = \lceil \log_2(\epsilon \delta n)\rceil$, and since $\epsilon < \frac{1}{32}$, $3\cdot 2^r \leq 6\epsilon \delta n \leq \frac{\delta n}{4}$. Thus, 
\begin{equation*}
\prob \{s(X_{t+1}, Y_{t+1}) \geq 2^r\} \geq 2^{r-2} \delta n^{-1}
\end{equation*}
Finally, the probability of $\mathscr{B}_r$ is the probability that $s(X_{t+1}, Y_{t+1})$ isn't at least $2^{r+1}$ for any $t \in I_r$, and therefore,
\begin{equation*}
\prob \{\mathscr{B}_r\} \leq \left(1-2^{r-2}\delta n^{-1}\right)^{a_{r-1}} \leq \exp(-2^{r-2}\delta n^{-1} a_{r-1})
\end{equation*}
Now, since $a_{r-1} = \lceil 2 \delta^{-1} 2^{-r+1} n (\log_2 n - r+1)\rceil$, 
\begin{equation*}
\prob \{\mathscr{B}_r\} \leq  e^{r-1 - \log_2 n}  \leq 2^{r-1-\log_2 n} = \frac{2^{r-1}}{n} 
\end{equation*}
using the fact that $r \leq K \leq \log_2 n + 1$, and hence $r - 1 - \log_2 n \leq 0$. Therefore, 
\begin{equation}\label{sBbound}
\prob \{\tilde{\mathscr{B}}_r\} \leq \sum_{r=j}^{K} \prob \{\mathscr{B}_r\} \leq \sum_{r=j}^{K} \frac{2^{r-1}}{n} \leq \frac{2^r}{n} \leq \epsilon \delta  + \frac{1}{n}
\end{equation}
Now, adding \eqref{sFbound}, \eqref{sHbound}, \eqref{sMbound} and \eqref{sBbound}, 
\begin{equation}
\prob \{s(X_t, Y_t) < 2^r\text{ for any } t \in I_r\text{, for } r\in [j, K]\} \leq 2^{11} \delta^{-1} \epsilon \left| \log(\epsilon \delta)\right|
\end{equation}
as required. \end{proof}

\begin{rmk}\label{constantlarge}
Assiduously tracking down all the constants in the above argument shows that the mixing time was  bounded above by $2^{25} n \log n$ or so. This, of course, is very far from the correct answer of $\frac{1}{2} n \log n$. While this argument can almost certainly be mildly tweaked to give a less intimidating answer such as $10 n \log n$, it is unlikely that it could be manipulated to give the right constant. 
\end{rmk}

\section{Technical Lemmas}\label{schrammtechnicallemmas}
In this section, the technical results in Lemmas \ref{schrammshrinklemma} through \ref{sgrowthlemma} are proved. For the convenience of the reader, the results are restated. 

\begin{lem14}
Let $\sigma$ be in $S_n$, and let $(\bar{X}_t)_{t\geq 1}$ be the random transposition walk starting at $\sigma$. Then, the expected number of $v$ such that $\left|C_1(v)\right| < \left|C_0(v)\right|$ and $\left|C_1(v)\right| < x$ is no greater than $\frac{x^2}{n}$. 
\end{lem14}

\begin{proof}[\bf Proof:]
Let $\sigma = (a_1, \dots, a_m)$. Clearly, the only way that $\left| C_1(v) \right| < \left| C_0(v) \right|$ is if the cycle containing $v$ is split; furthermore, the only way that $\left|C_1(v)\right| < x$ is if $v$ winds up in a piece of size less than $x$. The `ordered' splitting formula shows that the probability of splitting $a_i$ into $(r, a_i-r)$ is $\frac{a_i}{n^2}$. Consider the cases where either $r < x$ or $a_i -r < x$. Thus, summing over the possible $a_i$, 
\begin{align*}
&\expect \left| \{ v \text{ s.t. }\left| C_1(v) \right| <\left| C_0(v) \right|, \left| C_1(v) \right| < x\} \right| \leq \\
&\hspace{150 pt}\sum_{i=1}^m \left( \sum_{r = 1}^{x-1} r \cdot \frac{a_i}{n^2} + \sum_{r = a_i - x + 1}^{a_i - 1} (a_i - r) \cdot \frac{a_i}{n^2}\right) 
\end{align*}
It's clear that
\begin{align*}
\sum_{r = a_i - x + 1}^{a_i - 1} (a_i - r)\frac{a_i}{n^2} &=  \sum_{r = 1}^{x-1} r \cdot \frac{a_i}{n^2} = \frac{a_i}{n^2} \sum_{c=1}^{x-1} r \leq \frac{a_ix^2}{2n^2} 
\end{align*} 
Therefore, 
\begin{align*}
\expect \left| \{ v \text{ s.t. }\left| C_1(v) \right| <\left| C_0(v) \right|, \left| C_1(v) \right| < x\} \right| &\leq \sum_{i=1}^m  \frac{a_ix^2}{n^2}= \frac{x^2}{n^2} \sum_{i=1}^m a_i \\
&= \frac{x^2}{n^2} \cdot n = \frac{x^2}{n} 
\end{align*}
as required. \end{proof}

For the next four lemmas, let $(X_t, Y_t)$ be our usual coupling starting at $(\sigma, \tau)$, where $\rho(\sigma, \tau) = 1$, $s(\sigma, \tau) = b$ and $m(\sigma, \tau) = c$. For these proofs, it will be useful to reference the original definition of the coupling and the possible pairs $(X_1, Y_1)$ in Definition \ref{schrammcouplingdefn}. 

\begin{lem15}
If $x \leq c$, then 
\begin{equation*}
\prob\left\{m(X_1, Y_1) < x\right\} \leq \frac{2x^2}{n^2}. 
\end{equation*}
\end{lem15}

\begin{proof}[\bf Proof:]
Let us assume without loss of generality that 
\begin{align} \label{defps1}
\sigma &= (a_1, \dots, a_n, b, c) \\
\tau &= (a_1, \dots, a_n, b+c)\nonumber
\end{align}

Consider how $m(X_1, Y_1)$ could be smaller than $c$. Note that performing an operation involving only the $a_i$ on $\sigma$ and $\tau$, then $X_1$ and $Y_1$ will still differ in $b,c$ and $b+c$, so $m(X_1, Y_1) = c$. Furthermore, merging $a_i$ with $b$ in $\sigma$ and $a_i$ with $b+c$ in $\tau$, then $X_1$ and $Y_1$ will differ in the parts $(b+a_i, c, b+c+a_i)$, which are greater, respectively, than $(b, c, b+c)$. This means that $m(X_1, Y_1) \geq c$. Similar reasoning holds for merging $a_i$ with $c$ in $\sigma$, and hence these cases do not contribute to $\prob \{m(X_1, Y_1) < x\}$. 

Also, note that if $b$ is split into $\{r, b-r\}$ for $r \leq \frac{b}{2}$, then
\begin{align*}
X_1 &= (a_1, \dots, a_m, r, b-r,c) \\
Y_1 &= (a_1, \dots, a_m, r, b+c-r)
\end{align*}
Clearly, $c \geq b\geq b-r$, and therefore $m(X_1, Y_1) = c$. Thus, $m$ cannot decrease if $b$ is split in $\sigma$. This gives cases: splitting $c$ in $\sigma$, and merging $b$ and $c$ in $\sigma$. The cases in which the coupling meets can be ignored, since $m(\alpha, \alpha) = n \geq c$, and hence these cases do not contribute to $\prob \{m(X_1, Y_1) < x\}$. 

\paragraph{Splitting $c$ in $\sigma$:} If $c$ is split into $\{r, c-r\}$ for $r\leq \frac{c}{2}$, then 
\begin{align*}
X_1 &= (a_1, \dots, a_m, r, b, c-r)\\
Y_1 &= (a_1, \dots, a_m, r, b+c-r)
\end{align*}
Clearly, $m(X_1, Y_1) \geq c-r$. Thus, to have $m(X_1, Y_1) <x$, it must be that $c-r < x$, and thus $r > c-x$. By definition, $r\leq \frac{c}{2}$, and hence 
\begin{equation*}
c-x < r \leq \frac{c}{2}
\end{equation*}
If $2x < c$, this set contains no elements, so assume for now that $2x \geq c$. Then the number of possible $r$ is at most $\frac{c}{2} - (c-x) = \frac{2x - c}{2}$. Since the probability of splitting $c$ into $\{r, c-r\}$ is at most $\frac{2c}{n^2}$ for each $r$, 
\begin{equation}\label{msplitc}
\prob \left( m(X_1, Y_1) < x, \text{ $c$ split in $\sigma$} \right)
\leq \frac{c(2x-c)}{n^2} \leq \frac{x^2}{n^2}
\end{equation} 
using the AM-GM inequality and the assumption that $2x - c \geq 0$. Furthermore, the above inequality also holds when $2x < c$, since in that case, the left-hand side is $0$. 

\paragraph{Merging $b$ and $c$ in $\sigma$:}
If $b$ and $c$ are merged in $\sigma$,  
\begin{align*}
X_1 &= (a_1, \dots, a_m, b+c)\\
Y_1 &= (a_1, \dots, a_m, s, b+c-s)
\end{align*}
for some $s \leq \frac{b+c}{2}$. Hence, $m(X_1, Y_1) = b+c-r$. Again, to have $b+c - r < x$, it must be that $s > b+c-x$, and the probability of each split is at most $\frac{2(b+c)}{n^2}$. Thus, analogously to above, consider
\begin{equation*}
b+c -x < r \leq \frac{b+c}{2}
\end{equation*} 
and hence the total number of such $s$ is at most $\frac{2x - (b+c)}{2}$ if $2x \geq b+c$, and $0$ otherwise. Therefore, if $2x \geq b+c$, 
\begin{equation}\label{mmergebc}
P\{m(X_1, Y_1) < x, \text{ $b$ and $c$ merged in $\pi$}\} \leq \frac{(b+c)(2x - (b+c))}{n^2} \leq \frac{x^2}{n^2}
\end{equation} 
again using AM-GM. This clearly also holds for $2x < b+c$. 

Finally, adding \eqref{msplitc} and \eqref{mmergebc}, 
\begin{equation*}
P\{m(X_1, Y_1) < x\} \leq \frac{2x^2}{n^2}
\end{equation*}
as required. 
\end{proof}
\begin{lem16}
If $x \leq c$, and $| V_0(y)| \geq R$, then
\begin{equation*}
\prob \{m(X_1, Y_1) \geq x +y \} \geq  \frac{2c(R- 2c)}{n^2} 
\end{equation*}
\end{lem16}
\begin{proof}[\bf Proof:]
Consider both the possibilities that 
\begin{align}\label{case1}
\pi &= (a_1, \dots, a_m, b, c)\\
\sigma &= (a_1, \dots, a_m, b+c)\nonumber 
\end{align}
and that
\begin{align}\label{case2} 
\pi &= (a_1, \dots, a_m, b+ c)\\
\sigma &= (a_1, \dots, a_m, b, c)\nonumber 
\end{align}
with $b\leq c$, since $V_t(y)$ is defined for $X_t$ and not $Y_t$, and therefore the symmetry breaks down. Merging $c$ with an $a_i \geq y$ will result in $m(X_1, Y_1) = c + a_i \geq x + y$. To calculate the probability of such a merge, the sum of these $a_i$ is needed.

In both cases \eqref{case1} and \eqref{case2}, since $\sigma$ and $\tau$ agree on the $a_i$, 
\begin{equation}\label{lowerbd}
\sum_{a_i\geq y} a_i \geq |V_0(y)| - (b+c) \geq R - 2c
\end{equation}
using Remark \ref{Vtremark}. For case \eqref{case1}, merging $c$ and some $a_i \geq y$ in $\sigma$ gives
\begin{align*}
X_1 &= (a_1', \dots, a_{m-1}', b, c+a_i) \\
Y_1 &= (a_1', \dots, a_{m-1}', b+ c+ a_i)
\end{align*}
where $\{a_1', \dots, a_{m-1}'\} = \{a_1, \dots, a_m\}/\{a_i \}$. Clearly, $c+a_i \geq b$, and therefore $m(X_1, Y_1) = c+ a_i \geq x + y$. The probability of merging $c$ with $a_i$ in $\sigma$ is $\frac{2ca_i}{n^2}$, and thus 
\begin{equation*}
\prob \{m(X_1^\pi, X_1^\sigma) \geq x+y\}  \geq \sum_{a_i\geq y} \frac{2ca_i}{n^2} = \frac{2c}{n^2}\sum_{a_i\geq y} a_i \geq \frac{2c(R-2c)}{n^2}
\end{equation*}
using Equation \eqref{lowerbd} for the last inequality. Thus, in case \eqref{case1} the proof is finished. Furthermore, since Equation \eqref{lowerbd} is  symmetric for the cases \eqref{case1} and \eqref{case2}, the second case is completely analogous. 
\end{proof}

\begin{lem17}
If $x \leq b$, then 
\begin{equation*}
\prob \{s(X_1, Y_1) < x\} \leq \frac{4x^2}{n^2}
\end{equation*}
\end{lem17}

\begin{proof}[\bf Proof:]
For simplicity, assume without loss of generality that $\pi$ and $\sigma$ satisfy \eqref{defps1} above. In the same way as in Lemma \ref{mshrinklemma} above, any operations involving $a_i$ cannot make $s(X_1, Y_1)$ smaller than $b$. Thus, the operations that might produce $s(X_1, Y_1) < x$ involve either splitting $b$ in $\sigma$, splitting $c$ in $\sigma$, or merging $b$ and $c$ in $\sigma$. Consider these cases separately. In the same way as before, the cases where the coupling meets can be ignored. 

\paragraph{Splitting $b$ in $\sigma$:} Recall that if $b$ is split into $\{r, b-r\}$ for $r\leq \frac{b}{2}$, then 
\begin{align*}
X_1 &= (a_1, \dots, a_m, r, b-r,c) \\
Y_1 &= (a_1, \dots, a_m, r, b+c-r)
\end{align*}
Thus, $s(X_1, Y_1) = \min(b-r, c) = b-r$. To have $s(X_1, Y_1) < x$, $b -r <x$ is needed. Hence, consider $r$ such that 
\begin{equation*}
b-x < r \leq \frac{b}{2}
\end{equation*}
If $2x < b$, this set contains no elements, so assume $2x \geq b$. Clearly, the above set is of size at most $x - \frac{b}{2}$. The probability of splitting $b$ into $(s, b-s)$ is at most $\frac{2b}{n^2}$ for each $s\leq \frac{b}{2}$, and therefore 
\begin{equation}\label{splitb} 
\prob \{s(X_1, Y_1) < x, \text{ $b$ split in $\sigma$}\} \leq \frac{2b}{n^2}\left(x - \frac{b}{2}\right) = \frac{b(2x - b)}{n^2} \leq \frac{x^2}{n^2}
\end{equation}
using AM-GM and the assumption that $2x \geq b$ for the last inequality. This clearly also holds if $2x < b$, since in that case the left-hand side is $0$. 

\paragraph{Splitting $c$ in $\sigma$:} 
This calculation is very similar to the above. The probability that $c$ is split into $\{r, c-r\}$, where $r \leq \frac{c}{2}$ and  $c-r < x$ is needed. Again, consider
\begin{equation*}
c-x < r \leq \frac{c}{2}
\end{equation*}
and since the probability of a particular split is at most $\frac{2c}{n^2}$, assuming that $2x \geq c$, the total probability of all these cases is at most 
\begin{equation}\label{splitc}
\prob \{s(X_1, Y_1) < x, \text{ $c$ split in $\sigma$}\} \leq \frac{c(2x - c)}{n^2} \leq \frac{x^2}{n^2}
\end{equation}
which again holds trivially when $2x < c$. 

\paragraph{Merging $b$ and $c$ in $\sigma$:} Recall that merging $b$ and $c$ in $\sigma$ is coupled with splitting $b+c$ into $\{r, b+c-r\}$ in $\tau$, where each split in $\tau$ occurs with the probability that it has not already been coupled with a split of $b$ or $c$ in $\sigma$. Thus, in this case, 
\begin{align*}
X_1 &= (a_1, \dots, a_m, b+c) \\
Y_1 &= (a_1, \dots, a_m, r, b+c-r)
\end{align*}
Assuming as usual that $r \leq \frac{b+c}{2}$, $s(X_1, Y_1) = r$. Now calculate the probability that $r < x$. Define
\begin{equation*}
P_r = \prob \{\text{$b$ and $c$ merge in $\pi$, $b+c$ splits into $\left\{ r, b+c-r \right\}$ in $\sigma$}\}
\end{equation*}
and bound $P_r$ for various values of $r$. Consider three different cases: 
\begin{itemize}
\item $r < \frac{b}{2}$: In this case, splitting $b+c$ into $\{r,  b+c-r\}$ in $\tau$ is coupled with both spliting $b$ into $\{r, b-r\}$ in $\sigma$ and with splitting $c$ into $\{r, c - r \}$ in $\sigma$. Thus,
\begin{equation} \label{Prdefn} 
P_r = \frac{2(b+c)}{n^2} - \frac{2b}{n^2} - \frac{2c}{n^2} = 0
\end{equation} 

\item $\frac{b}{2} \leq r < \frac{c}{2}$: In this case, splitting $b+c$ into $\{r, b+c-r\}$ in $\tau$ is coupled with splitting $c$ into $\{r, c-r\}$ in $\sigma$. Thus, 
\begin{equation*}
P_r \leq  \frac{2(b+c)}{n^2} - \frac{2c}{n^2} = \frac{2b}{n^2}
\end{equation*} 

\item $\frac{c}{2} \leq r$: In this case, splitting $b+c$ into $\{r, b+c-r\}$ in $\tau$ isn't coupled with any splits in $\sigma$. Hence, 
\begin{equation*}
P_r \leq  \frac{2(b+c)}{n^2}
\end{equation*} 
\end{itemize} 
Therefore, the reasoning above shows
\begin{equation}\label{Psbound}
P_r \leq 
\begin{cases} 0 & r< \frac{b}{2}\\
              \frac{2b}{n^2} & \frac{b}{2} \leq r < \frac{c}{2} \\
              \frac{2b+2c}{n^2} & \frac{c}{2} \leq r \leq \frac{b+c}{2} 
\end{cases} \ \   \Longrightarrow \ \ P_r \leq  \frac{4r}{n^2}
\end{equation}
where the right-hand inequality uses the fact that $b \leq c$. Therefore,  
\begin{align}\label{mergebc}
\prob \{s(X_1, Y_1) < x \text{, $b$ and $c$ merge in $\sigma$}\}
              &= \sum_{r=1}^{x-1} P_r \leq \sum_{r=1}^{x-1} \frac{4r}{n^2}
              = 4\frac{x(x-1)}{2n^2} \nonumber \\
              &\leq \frac{2x^2}{n^2}
\end{align}
Adding Equations \eqref{splitb}, \eqref{splitc} and \eqref{mergebc} gives
\begin{equation*}
\prob \{s(X_1, Y_1) < x\} \leq \frac{4x^2}{n^2}
\end{equation*}
as required. 
\end{proof}
\begin{lem18}
If $x$ and $y$ satisfy $x \leq b < x+y \leq c$, and $| V_0(y)| \geq R$, 
\begin{equation*}
\prob \{s(X_1, Y_1) \geq x + y \} \geq  \frac{2b(R- 3x-3y)}{n^2} 
\end{equation*}
\end{lem18}
\begin{proof}[\bf Proof:]
Just like in Lemma \ref{mgrowthlemma}, consider the two possibilities that  
\begin{align}\label{scase1}
\sigma &= (a_1, \dots, a_m, b, c)\\
\tau &= (a_1, \dots, a_m, b+c)\nonumber 
\end{align}
and that
\begin{align}\label{scase2} 
\sigma &= (a_1, \dots, a_m, b+ c)\\
\tau &= (a_1, \dots, a_m, b, c)\nonumber 
\end{align}
since $V_t(y)$ depends on $X_t$ and not on $Y_t$. As in the previous lemma, in both cases \eqref{scase1} and \eqref{scase2}, 
\begin{equation}\label{aibound}
\sum_{a_i \geq y} a_i \geq |V_0(y)| - (b+c) \geq R - (b+c)
\end{equation}
so case \eqref{scase1} may be assumed. Identical arguments will apply for \eqref{scase2}.

There are two possible ways to have $s(X_1, Y_1) \geq x+y$: either $b$ can merge with an $a_i \geq y$ in $\sigma$, or $b$ and $c$ can merge in $\sigma$, while $b+c$ can be split into $\{r, b+c -r \}$ in $\tau$, where $r\geq x+y$. Consider those cases separately. 

\paragraph{Merging $b$ and $a_i \geq y$ in $\sigma$:} Note that if $b$ and $a_i$ are merged in $\sigma$, then 
\begin{align*}
X_1 &= (a_1', \dots, a_{m-1}', b + a_i, c)\\
Y_1 &= (a_1', \dots, a_{m-1}', b + a_i+ c)
\end{align*}
where $\{a_1', \dots, a_{m-1}' \} = \{a_1, \dots, a_m\}/\{a_i\}$. Therefore, $s(X_1, Y_1) = \min(b+a_i, c)$. Since $b\geq x$ and $a_i\geq y$, $b+a_i\geq x+y$. By assumption, $c \geq x + y$, and so $s(X_1, Y_1) \geq x + y$. 

The probability of $b$ merging with a particular $a_i$ is $\frac{2ba_i}{n^2}$, and using the bound in Equation \eqref{aibound}, 
\begin{align}\label{saibound}
\prob \{s(X_1, Y_1) \geq x+y, b \text{ merges with some $a_i$ in $\sigma$} \} &= \sum_{a_i\geq y} \frac{2ba_i}{n^2} 
                               = \frac{2b}{n^2}\sum_{a_i \geq y} a_i\nonumber\\
                              &\geq \frac{2b(R-(b+c))}{n^2}
\end{align}

\paragraph{Merging $b$ and $c$ in $\sigma$:} If $c < 2x + 2y$, it will later show that the above bound in Equation \eqref{saibound} suffices. Therefore, for this case, assume that $c \geq 2x + 2y$. Consider the probability of merging $b$ and $c$ in $\sigma$, while splitting $b+c$ in $\tau$ into $\{r, b+c - r\}$, where $r \geq x+y$. 

Let $P_r$ be defined as in Equation \eqref{Prdefn}. Now a lower bound on 
\begin{align*}\label{Prsumlowerbound}
\sum_{r \geq x+y} P_r &= \prob\{\text{merge $b$ and $c$ in $\sigma$}\} \\ & \qquad - \prob \{\text{merge $b$ and $c$ in $\sigma$, stay at $\tau$}\} 
-  \sum_{r < x+y} P_r 
\end{align*} 
is needed. The above equality follows because merging $b$ and $c$ in $\sigma$ is always either coupled with splitting $b+c$ into $\{r, b+c-r\}$ in $\tau$, or staying at $\tau$. Here is a lower bound for the right-hand side.

To start, $c \geq 2x + 2y > 2b$. By Equation \eqref{lowerboundmergingbandc},
\begin{equation*}
\prob \left\{\text{merge $b$ and $c$ in $\sigma$, stay at $\tau$}\right\} =\min\left(p, \frac{1}{n} \right)
\end{equation*}
where $p = \prob(\text{$b +c$ split into $\{b, c\}$ in $\tau$, staying at $\sigma$})$. Now, since $c > 2b$, 
\begin{align*}
p &= \prob(\text{$b+c$ split into $\{b, c\}$ in $\tau$})- \prob (\text{$c$ split into $\{b, c - b\}$ in $\sigma$}) \\
 &= \frac{2(b+c)}{n^2} - \frac{2c}{n^2} = \frac{2b}{n^2} 
\end{align*}
and hence, since $b \leq \frac{n}{2}$,  
\begin{equation*}
\prob \left\{\text{merge $b$ and $c$ in $\sigma$, stay at $\tau$}\right\} =\min\left(\frac{2b}{n^2}, \frac{1}{n} \right) = \frac{2b}{n^2}
\end{equation*}
Furthermore, since $x+y \leq \frac{c}{2}$, Equation \eqref{Psbound} above implies that if $r < x+y$ then $P_r \leq \frac{2b}{n^2}$, and therefore
\begin{equation*}
\prob \{\text{merge $b$ and $c$ in $\sigma$, stay at $\tau$}\} + \sum_{r < x+y} P_r \leq \frac{2b(x+y)}{n^2}
\end{equation*}
Thus, since the probability of merging $b$ and $c$ is $\frac{2bc}{n^2}$, 
\begin{align}\label{mergelowerbound}
\prob \{s(X_1, Y_1) \geq x+y, \text{$b$ and $c$ merge in $\sigma$}\} &= \sum_{r \geq x+y} P_r \geq \frac{2bc}{n^2} - \frac{2b(x+y)}{n^2} \nonumber \\
& = \frac{2b(c - x - y)}{n^2}
\end{align} 
Combining all this information, if $c < 2x + 2y$, then Equation \eqref{saibound} shows that 
\begin{equation*}
\prob\{ s(X_1, Y_1) \geq x + y \} \geq \frac{2b(R-(b+c))}{n^2} \geq \frac{2b(R- 3x - 3y)}{n^2} 
\end{equation*} 
using the fact that $b \leq x + y$. Furthermore, if $c \geq 2x + 2y$, then combining Equation \eqref{saibound} and \eqref{mergelowerbound},  
\begin{align*}
\prob \{s(X_1, Y_1) \geq x+y\} &\geq \frac{2b\left(R-(b+c)\right)}{n^2} + \frac{2b\left(c-x-y\right)}{n^2} \\
&\geq \frac{2b(R-2x - 2y)}{n^2}
\end{align*}
Hence, in either case $\prob \{s(X_1, Y_1) \geq x+y\}  \geq \frac{2b(R - 2x - 2y)}{n^2}$, completing the proof. 
\end{proof}

    \bibliographystyle{plain}

    \bibliography{mybib}

\begin{thebibliography}{10}

\bibitem{AlonSpencerBook}
Noga Alon and Joel~H. Spencer.
\newblock {\em The probabilistic method}.
\newblock Wiley-Interscience Series in Discrete Mathematics and Optimization.
  John Wiley \& Sons Inc., New York, 1992.
\newblock With an appendix by Paul Erd{\H{o}}s, A Wiley-Interscience
  Publication.

\bibitem{BerestyckiZeitouniSchramm}
Nathanael Berestycki, Oded Schramm, and Ofer Zeitouni.
\newblock Mixing times for random k-cycles and coalescence-fragmentation
  chains.
\newblock \url{http://arxiv.org/abs/1001.1894}.
\newblock [Online; accessed 4-July-2011].

\bibitem{BroderStrongStationary}
Andrei~Z. Broder.
\newblock Unpublished manuscript.
\newblock 1985.

\bibitem{BubleyDyer}
R.~Bubley and M.~Dyer.
\newblock Path coupling: A technique for proving rapid mixing in markov chains.
\newblock In {\em Proceedings of the 38th Annual Symposium on Foundations of
  Computer Science}, pages 223--, Washington, DC, USA, 1997. IEEE Computer
  Society.

\bibitem{DiaconisBook}
Persi Diaconis.
\newblock {\em Group representations in probability and statistics}.
\newblock Institute of Mathematical Statistics Lecture Notes---Monograph
  Series, 11. Institute of Mathematical Statistics, Hayward, CA, 1988.

\bibitem{DiaconisZeitouniPD1}
Persi Diaconis, Eddy Mayer-Wolf, Ofer Zeitouni, and Martin P.~W. Zerner.
\newblock The {P}oisson-{D}irichlet law is the unique invariant distribution
  for uniform split-merge transformations.
\newblock {\em Ann. Probab.}, 32(1B):915--938, 2004.

\bibitem{DiaconisSaloffComparisonGroups}
Persi Diaconis and Laurent Saloff-Coste.
\newblock Comparison techniques for random walk on finite groups.
\newblock {\em Ann. Probab.}, 21(4):2131--2156, 1993.

\bibitem{DiaconisandShah}
Persi Diaconis and Mehrdad Shahshahani.
\newblock Generating a random permutation with random transpositions.
\newblock {\em Z. Wahrsch. Verw. Gebiete}, 57(2):159--179, 1981.

\bibitem{DoeblinCoupling}
W.~Doeblin.
\newblock Espos{\' e} de la th{\' e}orie des cha{\^ i}nes simple constantes de
  {M}arkov {\` a} un nombre fini d'{\' e}tats.
\newblock {\em Rev. Math. Union Interbalkan}, 2:77--105, 1938.

\bibitem{GriffeathMaximalCoupling}
David Griffeath.
\newblock A maximal coupling for {M}arkov chains.
\newblock {\em Z. Wahrscheinlichkeitstheorie und Verw. Gebiete}, 31:95--106,
  1974/75.

\bibitem{VigodaNonMarkovian}
T.P. Hayes and E.~Vigoda.
\newblock A non-markovian coupling for randomly sampling colorings.
\newblock In {\em Foundations of Computer Science, 2003. Proceedings. 44th
  Annual IEEE Symposium on}, pages 618--627, oct. 2003.

\bibitem{JonassonOverlapping}
J.~Jonasson.
\newblock Mixing time bounds for overlapping cycles shuffles.
\newblock {\em Electronic Journal of Probability}, 16:1281--1295, 2011.

\bibitem{kovchegovburton}
Yevgeniy Kovchegov and Robert Burton.
\newblock Mixing times via super-fast coupling.
\newblock \url{arXiv:0912.2759v1}.
\newblock [Online; accessed 4-July-2011].

\bibitem{YuvalBook}
David~A. Levin, Yuval Peres, and Elizabeth~L. Wilmer.
\newblock {\em Markov chains and mixing times}.
\newblock American Mathematical Society, Providence, RI, 2009.
\newblock With a chapter by James G. Propp and David B. Wilson.

\bibitem{Lindvall}
Torgny Lindvall.
\newblock {\em Lectures on the coupling method}.
\newblock Dover Publications Inc., Mineola, NY, 2002.
\newblock Corrected reprint of the 1992 original.

\bibitem{MatthewsStationaryTranspositions}
Peter Matthews.
\newblock A strong uniform time for random transpositions.
\newblock {\em J. Theoret. Probab.}, 1(4):411--423, 1988.

\bibitem{YuvalBCNotes}
Yuval Peres.
\newblock Mixing for markov chains and spin systems.
\newblock \url{http://www.stat.berkeley.edu/users/peres/ubc.pdf}.
\newblock [Online; accessed 4-July-2011].

\bibitem{PitmanMaximalCoupling}
J.~W. Pitman.
\newblock On coupling of {M}arkov chains.
\newblock {\em Z. Wahrscheinlichkeitstheorie und Verw. Gebiete},
  35(4):315--322, 1976.

\bibitem{SchrammLargeCycles}
Oded Schramm.
\newblock Compositions of random transpositions.
\newblock {\em Israel Journal of Mathematics}, vol. 147:pp.221--243, 2005.

\bibitem{Thorisson}
Hermann Thorisson.
\newblock {\em Coupling, stationarity, and regeneration}.
\newblock Probability and its Applications (New York). Springer-Verlag, New
  York, 2000.

\end{thebibliography}

\end{document}